\documentclass[a4paper,leqno]{amsart}

\usepackage{amsmath}
\usepackage{amsthm}
\usepackage{amssymb}
\usepackage{amsfonts}
\usepackage[applemac]{inputenc}
\usepackage{mathrsfs}
\usepackage{pdfsync}
\usepackage{color} 
\usepackage{mathtools}

\usepackage{float}

\def\C{{\mathbb C}}
\def\R{{\mathbb R}}
\def\N{{\mathbb N}}
\def\T{{\mathbb T}}

\def\le{\leqslant}
\def\ge{\geqslant}

\newcommand{\im}{\mathrm{Im}}

\theoremstyle{plain}
\newtheorem{theorem}{Theorem}[section]
\newtheorem{lemma}[theorem]{Lemma}
\newtheorem{corollary}[theorem]{Corollary}
\newtheorem{proposition}[theorem]{Proposition}

\theoremstyle{definition}

\newtheorem{remark}[theorem]{Remark}
\newtheorem*{remark*}{Remark}

\numberwithin{equation}{section}

\begin{document}

\title[Gross-Pitaevskii model for exciton-polariton condensates]
{On a dissipative Gross-Pitaevskii-type model for exciton-polariton condensates}

\author[P. Antonelli]{Paolo Antonelli}

\author[P. Markowich]{Peter Markowich}

\author[R. Obermeyer]{Ryan Obermeyer}

\author[J. Sierra]{Jesus Sierra}

\author[C. Sparber]{Christof Sparber}

\address[P. Antonelli]
{Gran Sasso Science Institute, viale F. Crispi 7, 67100 L'Aquilla, Italy}
\email{paolo.antonelli@gssi.it}

\address[P. Markowich]
{CEMSE Division, King Abdullah University of Science and Technology, Box 4700, Thuwal 23955-6900, Saudi Arabia}
\email{peter.markowich@kaust.edu.sa}

\address[R.~Obermeyer]
{Department of Mathematics, Statistics, and Computer Science, M/C 249, University of Illinois at Chicago, 851 S. Morgan Street, Chicago, IL 60607, USA}
\email{roberm2@math.uic.edu}

\address[J. Sierra]
{Department of Information Engineering, Computer Science and Mathematics, University of L'Aquila, 67100 L'Aquila, Italy}
\email{jesus.sierra@gssi.it}

\address[C.~Sparber]
{Department of Mathematics, Statistics, and Computer Science, M/C 249, University of Illinois at Chicago, 851 S. Morgan Street, Chicago, IL 60607, USA}
\email{sparber@math.uic.edu}

\begin{abstract}
We study a generalized dissipative Gross-Pitaevskii-type model arising in the description of exciton-polariton condensates. We derive global in-time 
existence results and various a-priori estimates for this model posed on the one-dimensional torus. 
Moreover, we analyze in detail the long-time behavior of spatially homogenous solutions and their respective steady states and 
present numerical simulations in the case of more general initial data. We also study the convergence to the corresponding adiabatic 
regime, which results in a single damped-driven Gross-Pitaveskii equation. 
\end{abstract}

\date{\today}


\thanks{This publication is based on work supported by the NSF through grant no. DMS 1348092.
}
\maketitle

\section{Introduction}\label{sec:intro}

Exciton-polaritons are hybrid light and matter {\it quasi-particles} of bosonic type. They 
arise from the strong coupling of photons with the electromagnetic dipolar moment of excitons, i.e., electron-hole pairs in semiconductors; see \cite{BrMo} for a general introduction. 
The experimental realization of {\it Bose-Einstein condensates} (BECs) of such exciton-polaritons 
has triggered the emergence of an exciting field of physical and mathematical research; see, e.g., \cite{Kasp} for a description of 
experimental evidence of BEC in exciton-polaritons. In contrast to more classical BECs in 
ultracold atomic gases, exciton-polariton condensates can be produced at much higher temperatures, due to 
their lower effective mass. Being produced in semiconductor microcavities, exciton-polaritons are also of 
interest in that they provide an example of a BEC occurring in a solid-state system.

Moreover, this type of condensate has the crucial novelty of being an intrinsically {\it non-equilibrium system}. The latter is due to 
the finite lifetime of polaritons, which requires one to replenish the condensate continuously via 
optically injected high energy excitations. In turn, this implies that any (stable) stationary state results from a dynamical balance of pumping and losses.

To describe such systems, from a mathematical point of view, the simplest possible approach is based on a mean-field model of 
Gross-Pitaevskii type; cf. \cite{Car} for a broad overview. Such a model has been proposed in \cite{WoCa, WCC} and formally derived in \cite{HDTT} through a quantum-kinetic approach. 
It consists of a  generalized open-dissipative Gross-Pitaevskii equation for the macroscopic wave-function, $\psi$, of the polaritons, coupled to a simple 
rate equation for the exciton reservoir density, $n$. In one spatial dimension (valid for, e.g., micro-wires) and using non-dimensionalized units, 
the model reads as follows:
\begin{equation}\label{GP}
\left\{\begin{aligned}
&i\partial_t\psi=-\frac12\partial_x^2\psi+g|\psi|^2\psi+\lambda n\psi+\frac{i}{2}(Rn-\alpha)\psi, \\
&\varepsilon\partial_tn=P-(R|\psi|^2+\beta)n,
\end{aligned}\right.
\end{equation}
subject to initial data 
\begin{equation}\label{ini}
\psi|_{t=0}=\psi_0(x), \quad n|_{t=0}=n_0(x), \quad x\in \mathbb T.
\end{equation}
Above, $g>0$ denotes the strength of the (repulsive) self-interaction of the polaritons, $\lambda>0$ describes the coupling of the condensate with the reservoir, 
and $\beta, \alpha >0$ are the respective polariton and exciton loss rates. In actual experiments, one usually has $\beta \gg \alpha$, see \cite{Car}. In addition, $R>0$ is the rate of stimulated scattering from the reservoir to the condensate and $P>0$ is the exciton creation rate. For simplicity, the latter is assumed to be constant throughout the spatial 
domain, but the case of an $x$-dependent $P$ has also been considered, cf. \cite{BOM, CCFHKR}. Finally, we introduce a small dimensionless parameter 
$0<\varepsilon\le 1$ which in the limit $\varepsilon \to 0$ will allow us to derive an effective model in the {\it adiabatic regime}, see Section \ref{sec:adiabatic}.

In our analysis, we shall consider \eqref{GP} on the one-dimensional torus of length $|\mathbb T|$. This choice is not only mathematically convenient but also 
physically motivated by the fact that a stable condensate can only form in a spatially confined system. If, instead, we take $x\in \R$, an additional 
confining potential would need to be taken into account, which significantly complicates the mathematical analysis. We mention, however, that the restriction to 
one spatial dimension is purely for notational convenience, and that the majority of our results generalize in a straightforward way to dimensions two and three.

In the following, we shall be interested in deriving various analytical results for \eqref{GP}, in particular 
concerning existence and uniqueness of solutions, as well as their 
long time behavior. To gain more qualitative insight, we shall also perform several numerical simulations of the system \eqref{GP} and some of its simplifications. To be more precise, the rest of the paper is organized as follows: 

In Section \ref{sec:a-priori}, we start with a basic local in-time existence result for smooth solutions (the proof of this result is standard 
and can be found in Appendix \ref{sec:LWPH1}). We shall then derive several a-priori estimates which will allow us to conclude that these solutions indeed 
exist globally for $t\in [0, \infty)$. 
In Section \ref{sec:space}, we consider the particular case of spatially homogenous initial data. Under these circumstances, \eqref{GP} 
simplifies to a system of ordinary differential equations, which we shall analyze in detail. In particular, we explicitly determine the associated steady states and the qualitative behavior
of the solutions locally near to these equilibria. The case of more general initial data is then considered in Section \ref{sec:num}, where we shall perform several 
numerical simulations to determine the qualitative properties of solutions of \eqref{GP} and their respective long-time behavior. Finally, we shall 
study the limit $\varepsilon \to 0$ in Section \ref{sec:adiabatic}. In this limiting regime, the original system \eqref{GP} simplifies to a single damped-driven Gross-Pitaevskii 
equation for $\psi$. 

\subsection*{Acknowledgements}
The authors are grateful to the anonymous referee for helpful suggestions to improve upon an earlier version of this paper: In particular, we are grateful for 
pointing out the pointwise $L^\infty$-bound on $n$ (see Lemma \ref{lem:n_inf}) and for suggesting a Lyapunov-type functional similar to the one introduced in Proposition \ref{prop:energy}.


\section{Existence of smooth global in-time solutions}\label{sec:a-priori}

\subsection{Local in-time existence and basic a-priori estimates}

We start with the following result which establishes existence and uniqueness for smooth solutions of \eqref{GP}, locally in-time. 
The proof follows by a standard fixed point argument, which for the sake of completeness will be given in Appendix \ref{sec:LWPH1}.

\begin{proposition}\label{prop:local} 
Let $(\psi_0, n_0) \in \mathcal H^s(\T)$ for $s > 1/2$, where 
$$\mathcal{H}^s (\T):= H^s(\T) \oplus H^s(\T),\quad s\in \N.$$
 Then there exists a time $T_{\mathrm{max}} > 0$ 
and a unique local in-time solution $(\psi, n) \in C([0,T_{\mathrm{max}});\mathcal{H}^s(\T))$ of \eqref{GP}, depending continuously on the initial data. 
Furthermore, the solution is 
maximal in the sense that if $T_{\mathrm{max}} < +\infty$, then 
\begin{equation}\label{blow}
\lim_{t \rightarrow T_{\mathrm{max}}} \big (\lVert \psi(t, \cdot) \rVert_{{H}^s} + \lVert n(t, \cdot) \rVert_{{H}^s}\big)= +\infty.
\end{equation}
\end{proposition}

Classical arguments (see, e.g. \cite{Caz} for more details) imply that the existence time, $T_{\rm max}>0$, does not depend on the choice of Sobolev index $s>1/2$, i.e., 
we have {\it persistence of regularity} on the time-interval $[0, T_{\rm max})$.

For solutions $(\psi, n) \in C ([0, T_{\rm max}); \mathcal H^s(\T))$, with $s>1/2$, we define the {\it total mass}, $M(t)$, as the sum of the individual masses of the condensate and reservoir, the latter weighted by $\varepsilon\in (0,1]$, i.e.
\[ 
M(t)  = M_{\rm c}(t) + \varepsilon M_{\rm r}(t),
\]
where
\begin{equation}\label{M}
M_{\rm c}(t) = \int_{\T} | \psi(t,x)|^2 \, dx, \qquad M_{\rm r}(t) = \int_\T n(t,x) \, dx.
\end{equation}
Note that $M_r$ is well-defined, since by Cauchy-Schwarz
\[
\int_\T |n (t,x)| \, dx \le \lVert n (t, \cdot) \rVert_{L^2} \sqrt{|\T|} < \infty.
\]
That $n(t,x)$ can indeed be interpreted as a (positive) mass-density is guaranteed by the following lemma.

\begin{lemma}\label{lem:n_pos}
Let $(\psi, n) \in C ([0, T_{\rm max}); \mathcal H^1(\T))$ be a solution of \eqref{GP} with initial data $n_0(x)>0$. Then $n(t,x)>0$ for all $t\in\left[0,T_{\mathrm{max}}\right)$.
\end{lemma}
\begin{proof}
Sobolev imbedding implies that for $\psi \in C ([0, T_{\rm max}); \mathcal H^1(\T))$, we indeed have $\psi \in C([0, T_{\rm max})\times \T)$. 
In view of the second equation of \eqref{GP}, we therefore have that $n(t,x)$ is continuously differentiable with respect to $t$, uniformly in $x\in \T$. 
The result then follows directly from the fact that $P>0$ and the usual variation of constants formula:
\[
n(t,x) = n_0(x) e^{- \int_0^t \Gamma(\tau, x) \, d\tau} + \frac{P}{\varepsilon} \int_0^t e^{- \int_s^t \Gamma(\tau, x) \, d\tau}\, ds,
\]
where $\Gamma(t,x) = \left( R|\psi(t, x)|^2 +\beta \right)/\varepsilon$, for $\varepsilon>0$. 
\end{proof}

Next, we shall prove an a-priori bound on the total mass. Notice that this estimate and the ones to follow are uniform in $\varepsilon$.
\begin{lemma}\label{lem:mas}
Let $(\psi, n) \in C ([0, T_{\rm max}); \mathcal H^1(\T))$ be a solution of \eqref{GP}. Then, its total mass, $M(t)$, is uniformly bounded. More precisely, we 
have 
\[
M(t) \le e^{-\gamma t}\left(M(0) - \frac{P |\T|}{\gamma} \right) + \frac{P |\T|}{\gamma}, \quad \forall \, 0\le t < T_{\rm max}, \quad \varepsilon>0,
\] 
where $\gamma = \min \{\alpha, \beta\}$. In the case where $\alpha =\beta$ and $\varepsilon=1$, this estimate becomes an equality and thus, if $T_{\rm max} = +\infty$, we find
\[
\lim_{t\to +\infty} M(t) = \frac{P |\T|}{\beta}\equiv \frac{P |\T|}{\alpha}.
\]
\end{lemma}

\begin{proof} Below, we assume that the initial data is sufficiently smooth, say $\psi_0, n_0 \in H^3(\T)$. In view of Proposition \ref{prop:local}, this yields a 
solution $$(\psi, n)\in C([0, T_{\rm max});\mathcal H^3 (\T))\cap C^1((0, T_{\rm max});\mathcal H^1 (\T))$$ 
for which all subsequent computations are rigorously justified. Invoking 
a standard density argument (see, e.g., \cite{Tao}) combined with the continuous dependence on initial data (and the asserted persistence of regularity), 
we can conclude that the result holds for $\mathcal H^1$-solutions.

Multiplying the first equation in \eqref{GP} by $\overline{\psi}$, integrating over $\T$, and taking the real part, we obtain
\[
\frac{d}{dt}M_{\rm c}(t) = \int_\T (Rn-\alpha)|\psi|^2 dx.
\]
Similarly, integrating the second equation in \eqref{GP} over $\T$ gives
\[
\varepsilon\frac{d}{dt} M_{\rm r}(t) = \int_\T P - (R|\psi|^2+\beta)n \, dx.
\]
Therefore, we have
\begin{align*}
\frac{d}{dt}M(t) &= \int_\T (Rn-\alpha)|\psi|^2 dx + P - (R|\psi|^2+\beta)n \, dx \\
&= \int_\T P dx - \alpha \int_\T |\psi|^2 dx - \beta \int_\T n \, dx \\
&\le P |\T| - \gamma M(t),
\end{align*}
with $\gamma = \min \{\alpha, \beta\}$. Integrating in time, yields
\begin{align*}
M(t) &\le e^{-\gamma t}M(0) + P |\T| \int_0^t e^{-\gamma(t-s)} ds \\
&= e^{-\gamma t}\left(M(0) - \frac{P |\T|}{\gamma} \right) + \frac{P |\T|}{\gamma} .
\end{align*}
In the case where $\alpha =\beta$ and $\varepsilon=1$, we see that 
this inequality actually becomes an equality. 
\end{proof}

This uniform in-time estimate on $M(t)$ is not sufficient to conclude $T_{\rm max} = +\infty$. However, it shows that the only obstruction to global existence for 
solutions $(\psi, n)(t, \cdot) \in \mathcal H^1$ is the possibility that 
\begin{equation*}
\lim_{t \to {T_{\rm max}}} \big (\| \partial_x  \psi (t, \cdot )\|_{L^2} + \| \partial_x  n (t, \cdot )\|_{L^2} \big)=+\infty.
\end{equation*}
To rule out this scenario, we shall, in a first step, derive a point-wise estimate on the reservoir density $n(t,x)$ below.

\begin{lemma}\label{lem:n_inf}
Let $\left(\psi,n\right)\in C\left(\left[0,T_{\mathrm{max}}\right);\mathcal{H}^{1}\left(\mathbb{T}\right)\right)$
be a solution of \eqref{GP}. Then,
\begin{equation*}
n^{2}\left(t,\cdot\right)\le e^{-t\beta/\varepsilon}\left(n_{0}^{2}(\cdot)-\frac{P^{2}}{\beta^{2}}\right)+\frac{P^{2}}{\beta^{2}},\hspace*{1em}\forall \, 0\le t< T_{\mathrm{max}}.\label{eq:unif_b_n}
\end{equation*}
\end{lemma}

\begin{proof} We again assume that the solution pair $\left(\psi,n\right)$
is sufficiently smooth to justify the computations below and then argue by density.
Multiplying the second equation in \eqref{GP} by $2n$ we obtain
\[
\varepsilon\partial_{t}n^{2}=2Pn-2\left(R\left|\psi\right|^{2}+\beta\right)n^{2}\le\frac{P^{2}}{\beta}-\beta n^{2},
\]
where the last inequality is a consequence of $\left(\frac{P}{\sqrt{\beta}}-\sqrt{\beta}n\right)^{2}\ge0$.
This implies that
\[
\varepsilon\partial_{t}\left(e^{t\beta/\varepsilon}n^{2}\right)\le e^{t\beta/\varepsilon}\frac{P^{2}}{\beta},
\]
and the result then follows directly by integrating this expression in time.
\end{proof}

In particular, this implies that $n(t, \cdot)\in L^{\infty}(\T)$ for all $t\in\left[0,T_{\mathrm{max}}\right)$, a fact we shall use in our energy estimates below. 
Note that this estimate for $n$ is {\it uniform} in $\varepsilon$ in the sense that $n(t, \cdot)$ remains bounded in $L^{\infty}(\T)$ in the limit $\varepsilon \to 0$.

\begin{remark} In addition, we infer that $n(t, \cdot)\in L^{p}(\T)$ for all $1\le p\le\infty$, $t\in\left[0,T_{\mathrm{max}}\right)$ and that
\[
\underset{t\rightarrow\infty}{\limsup} \, n(t,\cdot)\le\frac{P}{\beta},
\]
which should be compared with the result of Lemma \ref{lem:mas}.
\end{remark}


\subsection{Global smooth solutions} 
To obtain global in-time existence of smooth solutions, we consider the following {\it energy-type functional}:
\[
F\left(t\right):=E\left(t\right)+\frac{\varepsilon}{2}\int_{\mathbb{T}}\left(\partial_{x}\sqrt{n}\right)^{2}dx
-\frac{\lambda P}{R}\int_{\mathbb{T}}\ln n\, dx+\frac{\beta\lambda}{R}\int_{\mathbb{T}}n\, dx,
\]
where
\[
E\left(t\right):=\int_{\mathbb{T}}e (t,x)\, dx,
\]
with energy density
\[
e(t,x):=\frac{1}{2}\left|\partial_{x}\psi\right|^{2}+\frac{g}{2}\left|\psi\right|^{4}+\lambda n\left|\psi\right|^{2}.
\]

\begin{proposition}\label{prop:energy}
Let $(\psi, n) \in C ([0, T_{\rm max}); \mathcal H^1(\T))$ be a solution of \eqref{GP} such that $n(t, \cdot)>0$ for all $t\in [0, T_{\rm max})$. 
Then there exist non-negative constants, $C_1, C_2$, such that
\[
F\left(t\right)\le e^{C_{1}t}\left(F({0})+\frac{C_{2}}{C_{1}}\right)-\frac{C_{2}}{C_{1}}<\infty,
\]
for all $t\in\left[0,T_{max}\right)$. 
\end{proposition}

\begin{proof} As before, we first consider sufficiently smooth solutions $\psi, n$ and then argue by density to extend our computations to 
solutions $(\psi, n) (t, \cdot)\in \mathcal H^1(\T)$. Differentiating $E(t)$ and using \eqref{GP}, we obtain, after some straightforward computations 
\begin{align*}
\frac{d}{dt}E\left(t\right)= & \int_{\mathbb{T}}\left(Rn-\alpha\right)e(t,x)\, dx-\frac{1}{4}\int_\T \left(Rn-\alpha\right)\partial_{x}^{2}\left|\psi\right|^{2}dx+\frac{g}{2}\int(Rn-\alpha)|\psi|^4\,dx\nonumber \\
 & +\frac{\lambda}{\varepsilon}\int_{\mathbb{T}}\left(P-\left(R\left|\psi\right|^{2}+\beta\right)n\right)\left|\psi\right|^{2}dx\nonumber \\
= & \int_{\mathbb{T}}\left(Rn-\alpha\right)e\left(t,x\right)dx\nonumber + \frac{R}{4}\int_{\mathbb{T}}\left(\partial_{x}n\right)\partial_{x}\left|\psi\right|^{2}dx +\frac{g}{2}\int(Rn-\alpha)|\psi|^4\,dx \\
 & +\frac{\lambda}{\varepsilon}\int_{\mathbb{T}}\left(P-\left(R\left|\psi\right|^{2}+\beta\right)n\right)\left|\psi\right|^{2}dx,
\end{align*}
where the second equality follows from integration by parts and the fact that $\int_\T \partial_x^2 |\psi|^2 dx = 0$.
Differentiating the second equation in \eqref{GP} w.r.t. $x$, yields
\[
\varepsilon\partial_{t}\partial_{x}n=-\left(R\left|\psi\right|^{2}+\beta\right)\partial_{x}n-Rn\partial_{x}\left|\psi\right|^{2},
\]
which directly implies that 
\[
\varepsilon\partial_{t}\left(\partial_{x}n\right)^{2}=-2\left(R\left|\psi\right|^{2}+\beta\right)\left(\partial_{x}n\right)^{2}-2Rn\partial_{x}\left|\psi\right|^{2}\left(\partial_{x}n\right).
\]
Using this identity and the fact that $n>0$ by assumption,
we can compute 
\begin{align*}
\varepsilon\partial_{t}\frac{\left(\partial_{x}n\right)^{2}}{n}= & -\varepsilon\left(\partial_{t}n\right)\frac{\left(\partial_{x}n\right)^{2}}{n^{2}}+\frac{\varepsilon}{n}\partial_{t}\left(\partial_{x}n\right)^{2}\\
= & -P\frac{\left(\partial_{x}n\right)^{2}}{n^{2}}-\left(R\left|\psi\right|^{2}+\beta\right)n\frac{\left(\partial_{x}n\right)^{2}}{n^{2}}-2R\partial_{x}\left|\psi\right|^{2}\left(\partial_{x}n\right).
\end{align*}
Therefore,
\[
\frac{R}{4}\partial_{x}\left|\psi\right|^{2}\left(\partial_{x}n\right)=-\frac{P}{8}\frac{\left(\partial_{x}n\right)^{2}}{n^{2}}-\frac{1}{8}\left(R\left|\psi\right|^{2}+\beta\right)n\frac{\left(\partial_{x}n\right)^{2}}{n^{2}}-\frac{\varepsilon}{8}\partial_{t}\frac{\left(\partial_{x}n\right)^{2}}{n}.
\]
Plugging this into the expression of the time derivative of $E(t)$ obtained above, and keeping in mind 
that $\left(\partial_{x}n\right)^{2}=4n\left(\partial_{x}\sqrt{n}\right)^{2}$,
we find the following identity:
\begin{align*}
& \frac{d}{dt}\left(E\left(t\right)+\frac{\varepsilon}{2}\int_{\mathbb{T}}\left(\partial_{x}\sqrt{n}\right)^{2}dx\right)=  
\int_{\mathbb{T}}\left(Rn-\alpha\right)e (t,x)\, dx \, +\frac{g}{2}\int(Rn-\alpha)|\psi|^4\,dx \\
 & \, +\frac{\lambda}{\varepsilon}\int_{\mathbb{T}}\left(P-\left(R\left|\psi\right|^{2}+\beta\right)n\right)\left|\psi\right|^{2}\, dx
  -\frac{1}{8}\int_{\mathbb{T}}\left(P+\left(R\left|\psi\right|^{2}+\beta\right)n\right)\frac{\left(\partial_{x}n\right)^{2}}{n^{2}}\, dx.
\end{align*}
On the other hand, 
\[
\frac{\lambda}{\varepsilon}\int_{\mathbb{T}}\left(P-\left(R\left|\psi\right|^{2}+\beta\right)n\right)\left|\psi\right|^{2}dx=  \lambda\int_{\mathbb{T}}\left(\partial_{t}n\right)\left|\psi\right|^{2}dx,
\]
in view of the second equation in \eqref{GP}. The latter also implies that
\[
\left|\psi\right|^{2}= \frac{1}{R}\left(\frac{-\varepsilon\partial_{t}n+P}{n}-\beta\right),
\]
and thus, we can rewrite
\begin{align*}
& \frac{\lambda}{\varepsilon}\int_{\mathbb{T}}\left(P-\left(R\left|\psi\right|^{2}+\beta\right)n\right)\left|\psi\right|^{2}dx= \\
 & = -\frac{\varepsilon\lambda}{R}\int_{\mathbb{T}}\frac{\left(\partial_{t}n\right)^{2}}{n}\, dx+\frac{\lambda P}{R}\int_{\mathbb{T}}\frac{\partial_{t}n}{n}\, dx -\frac{\beta\lambda}{R}\int_{\mathbb{T}}\partial_{t}n \, dx\\
 & = -\frac{\varepsilon\lambda}{R}\int_{\mathbb{T}}\frac{\left(\partial_{t}n\right)^{2}}{n}\, dx+\frac{\lambda P}{R}\frac{d}{dt}\int_{\mathbb{T}}\ln n\, dx -\frac{\beta\lambda}{R}\frac{d}{dt}\int_{\mathbb{T}}n \, dx.
\end{align*}
In summary, this yields
\begin{align*}
& \frac{d}{dt}\left(E\left(t\right)+\frac{\varepsilon}{2}\int_{\mathbb{T}}\left(\partial_{x}\sqrt{n}\right)^{2} dx-\frac{\lambda P}{R}\int_{\mathbb{T}}\ln n \, dx+\frac{\beta\lambda}{R}\int_{\mathbb{T}}n \, dx\right)\\
& =\int_{\mathbb{T}}\left(Rn-\alpha\right)e(t,x)dx-\frac{\varepsilon\lambda}{R}\int_{\mathbb{T}}\frac{\left(\partial_{t}n\right)^{2}}{n}dx
-\frac{1}{8}\int_{\mathbb{T}}\left(P+\left(R\left|\psi\right|^{2}+\beta\right)n\right)\frac{\left(\partial_{x}n\right)^{2}}{n^{2}}dx\\
& +\frac{g}{2}\int(Rn-\alpha)|\psi|^4\,dx \le 2R \int_{\mathbb T} n e(t,x) \, dx,
\end{align*}
since all other terms on the right hand side are non-positive. 
Having in mind the definition of $F(t)$ and using the fact that $n(t, \cdot) \in L^{\infty}(\T)$, cf. Lemma \ref{lem:n_inf}, this implies 
\[
\frac{d}{dt}F\left(t\right) \le 2R \| n(t, \cdot) \|_{L^\infty} \, E(t)\le C_1 \left(F(t) + \frac{\lambda P}{R}\int_{\mathbb T}( \ln n)_+ \, dx\right), 
\]
where $( \ln n)_+= \max\{\ln n, 0\}\le n$. Using the $L^\infty$-bound on $n$ one more time, then allows us to bound
\[
\frac{d}{dt}F\left(t\right)\le C_{1}F\left(t\right) +C_2,
\]
where $C_2=C_2(\| n(t, \cdot) \|_{L^\infty})\ge0$. 
Integrating this last inequality with respect to time then gives the asserted result.
\end{proof}

The exponential bound obtained for $F(t)$ is most likely far from optimal. Nevertheless, it is sufficient to conclude global in-time existence:

\begin{theorem}\label{thm:global}
Let $(\psi_0,n_0)^\top \in \mathcal{H}^1$, with $n_0>0$. Then there exists a unique global in-time solution 
$(\psi, n) \in C ([0, \infty); \mathcal H^1(\T))$ of the system \eqref{GP}. In addition, its total mass, $M(t)$, 
is uniformly bounded for all $t\ge 0$.
\end{theorem}

\begin{proof}
From Proposition \ref{prop:local}, we know that for $(\psi_0,n_0) \in \mathcal{H}^1$, we obtain a 
unique maximal solution in $\mathcal H^1$ obeying the blow-up alternative \eqref{blow}. 
Recall that, in view of Lemma \ref{lem:mas}, we have a uniform bound on both $\| \psi(t, \cdot)\|_{L^2}$ and $\| n(t, \cdot)\|_{L^2}$, 
and thus it only remains to control the derivative of both $\psi$ and $n$ w.r.t. $x$ in $L^2(\mathbb T)$.

To this end, Lemma \ref{lem:n_pos} ensures $n(t,x)>0$ for all $t\in\left[0,T_{\mathrm{max}}\right)$, 
and thus we can apply Proposition \ref{prop:energy} to conclude that $F(t)$ remains bounded for all $t\in\left[0,T_{\mathrm{max}}\right)$.  
Together with the $L^\infty$-bound on $n$ established in Lemma \ref{lem:n_inf} this implies, that 
\[
E(t) \le F(t) + C_2 \le {\rm const},\quad \forall t\in\left[0,T_{\mathrm{max}}\right),
\]
and since $g, \lambda >0$ we infer that $\| \psi(t, \cdot)\|_{H^1}$ is bounded for all $t\in\left[0,T_{\mathrm{max}}\right)$.

In addition, the fact that $n>0$ allows us to bound
\[
\int_\T (\partial_x n)^2 \, dx  \le  \| n \|_{L^\infty} \int_\T \frac{(\partial_x n)^2 }{n} \, dx =  4 \| n \|_{L^\infty} \int_\T (\partial_x\sqrt{ n})^2 \, dx \le {\rm const.},
\]
in view of Proposition \ref{prop:energy}. 

Continuity then implies that the $H^1$-norm of both $\psi$ and $n$ remain bounded as $t\to T_{\rm max}$. In turn this yields $T_{\rm max}=+\infty$, for otherwise 
we would have a contradiction to the maximality of $T_{\rm max}$. 
\end{proof}

\begin{remark}
As mentioned before, the energy estimate obtained in Proposition \ref{prop:energy} is far from optimal. 
In particular, it is not strong enough to study the existence of a global attractor of the system \eqref{GP}. 
We are currently investigating the possibility of applying local smoothing methods to obtain 
the uniform energy estimates needed in this case. This approach has been successfully used in, e.g., \cite{Com, Com2, EMNT, ET}.
\end{remark}

In the next section, we shall obtain a qualitative insight into the solutions of \eqref{GP} in the particular case of $x$-independent initial data.


\section{The case of space-homogenous solutions} \label{sec:space} 

\subsection{Asymptotic behavior of spatially homogenous solutions} 
In this section, we study the long-time behavior of solutions of \eqref{GP} with $\varepsilon=1$ and in the case of {\it spatially homogenous} initial data. 
To this end, it is convenient to 
rewrite \eqref{GP} into its fluid-dynamical form, using $\psi=\sqrt{\rho} e^{i \phi}$. In this way, one formally obtains
\begin{equation}\label{QHD}
\left\{\begin{aligned}
&\partial_t\rho+\partial_x(\rho\partial_x\phi)=(Rn-\alpha)\rho,\\
&\partial_t\phi+\frac12(\partial_x \phi)^2+g\rho+\lambda n=\frac12\frac{\partial_x^2\sqrt{\rho}}{\sqrt{\rho}},\\
&\partial_tn=P-(R\rho+\beta)n.
\end{aligned}\right.
\end{equation}
For solutions which are $x$-independent, this Euler-type model simplifies considerably. Indeed, we obtain the following 
coupled system of ordinary differential equations for the condensate and reservoir densities:
\begin{equation}\label{eq:ode}
\left\{\begin{aligned}
\dot\rho=&(Rn-\alpha)\rho, \\
\dot n=&P-(R\rho+\beta)n,
\end{aligned}\right.
\end{equation}
subject to initial data 
\[
\rho|_{t=0} = \rho_0>0, \quad n|_{t=0} = n_0>0.
\]
When deriving the system \eqref{QHD} by means of the WKB ansatz $\psi=\sqrt{\rho}e^{i\phi}$, one usually faces the obstacle of 
possible vacuum regions. However, here we only consider spatially homogeneous solutions, so \eqref{eq:ode} is indeed 
completely justified and equivalent to \eqref{GP}.
\begin{lemma}
For any $\rho_0, n_0>0$, there exists a unique $(\rho, n)\in C^1([0, \infty), \R^2_+)$, solution of \eqref{eq:ode}, satisfying $\rho(t)>0$, $n(t)>0$, for all $t\ge0$.
\end{lemma}

Of course this result can be seen as a simple consequence of Theorem \ref{thm:global}. Its proof however, can be stated independently and 
reveals new estimates 
for $\rho(t)$ and $n(t)$.

\begin{proof}
Since the right hand side of \eqref{eq:ode} is quadratic (and thus locally Lipschitz) in $\rho, n$, a classical theorem implies 
existence of a unique local solution $(\rho, n)\in C^1([0, \tau), \R^2)$, for some $\tau>0$. Continuity also implies positivity of this solution. 
Because of that, the second line of  \eqref{eq:ode} allows us to estimate $\dot n \le P$, 
and thus 
\[
n(t) \le Pt +n_0, \quad \forall \, 0 \le t <\tau.
\]
Plugging this into the equation for $\rho$ gives
\[
\dot \rho \le R(Pt +n_0)\rho - \alpha \rho,
\]
which can be directly integrated, to yield 
\[
\rho(t) \le \rho_0 e^{RPt^2/2+ t(Rn_0 - \alpha)}, \quad \forall \, 0 \le t <\tau.
\]
In turn, this implies that the local 
solution $(\rho, n)(t)$ can be (uniquely) extended for all $t\ge0$.
\end{proof}

Given a solution $(\rho, n)$ of \eqref{eq:ode}, the condensate phase-function $\phi(t)$ associated to $\psi = \sqrt{\rho} e^{i \phi}$ can then 
be determined a-posteriori via
\begin{equation*}
\dot\phi=-g\rho-\lambda n, \quad \phi|_{t=0} = \phi_0,
\end{equation*} 
which gives
\[
\phi(t)= -\int_0^t  g \rho(\tau) + \lambda n(\tau) \, d\tau.
\]
If we set $\psi(t)=\sqrt{\rho}(t)e^{i\phi(t)}$, then we have defined a global in time, spatially homogeneous solution $(\psi, n)$ of \eqref{GP}.

\begin{remark} It has been (formally) shown in \cite{CCFHKR}, that small perturbations of spatially homogenous steady states (see subsection below) obey the Korteweg-de Vries equation, 
and thus admit solutions of dark-soliton type. It 
would be interesting to study the stability of these solitons within the dynamics of \eqref{GP}, but this is beyond the scope of the current article. 
\end{remark}


\subsection{Characterization of spatially homogenous equilibria} 

Now we turn our attention to the equilibrium points of the ODE system \eqref{eq:ode}, in the hope that they will give us 
some insight into the full ($x$-dependent) dynamics of \eqref{GP}.

A preliminary formal analysis of homogeneous stationary states, together with their stability properties, was already performed in \cite{WoCa, BOM}.

\begin{theorem}\label{thm:ODE}

The system \eqref{eq:ode} has two equilibrium points, given by
\begin{equation}
\xi_{1}=\left(\begin{array}{c}
\rho_{1}^{*}\\
n_{1}^{*}
\end{array}\right)=\left(\begin{array}{c}
\frac{1}{\alpha R}\left(PR-\alpha\beta\right)\\
\frac{\alpha}{R}
\end{array}\right),\label{eq:equi1}
\end{equation}
and
\begin{equation}
\xi_{2}=\left(\begin{array}{c}
\rho_{2}^{*}\\
n_{2}^{*}
\end{array}\right)=\left(\begin{array}{c}
0\\
\frac{P}{\beta}
\end{array}\right).\label{eq:equi2}
\end{equation}
Furthermore:
\renewcommand{\labelenumi}{(\roman{enumi})}
\begin{enumerate}
\item Both $\xi_{1}$ and $\xi_{2}$ are hyperbolic, except for the case
$PR-\alpha\beta=0$.
\item $\xi_{1}$ is an asymptotically stable spiral if $0<\frac{P^{2}R^{2}}{4\alpha^{2}}<PR-\alpha\beta$.
\item $\xi_{1}$ is an asymptotically stable node if $0<PR-\alpha\beta\le\frac{P^{2}R^{2}}{4\alpha^{2}}$.
\item $\xi_{1}$ is a saddle point, and hence unstable, if $PR-\alpha\beta<0$.
\item $\xi_{2}$ is a saddle point if $PR-\alpha\beta>0$.
\item $\xi_{2}$ is an asymptotically stable node if $PR-\alpha\beta<0$.
\end{enumerate}
\end{theorem}

From the physics point of view, the two equilibria $\xi_1, \xi_2$, have very different interpretations: $\xi_2$ corresponds to the 
case where {\it no condensate} is formed and the system simply relaxes to the stationary state $\frac{P}{\beta}$ for the reservoir. $\xi_1$, however,
describes a configuration with a {\it non-zero condensate} in dynamical equilibrium with the reservoir. It is thereby 
natural to impose the condition $PR-\alpha\beta>0$, in order to ensure that the 
equilibrium condensate density $\rho_1^*$ is positive. 

\begin{proof}
The fact that $\xi_{1}$ and $\xi_{2}$ are equilibrium points of
\eqref{eq:ode} follows immediately. For the remaining assertions on the qualitative behavior of these equilibria we shall 
use the well-known Hartman-Grobman theorem, see, e.g., \cite{Per}. The latter allows one to describe the 
local behavior of dynamical systems in the neighborhood of a hyperbolic equilibrium point via its linearization.

To this end, we translate $\xi_{1}$ to the origin using the following change of variables in \eqref{eq:ode}:
\begin{align*}
x=\rho-\frac{1}{\alpha R}\left(PR-\alpha\beta\right),\quad y=n-\frac{\alpha}{R}.
\end{align*}
Then \eqref{eq:ode} becomes
\begin{equation}\label{eq:sys_trans}
\left\{\begin{aligned}
\dot x& =  \left(R\left(y+\frac{\alpha}{R}\right)-\alpha\right)\left(x+\frac{1}{\alpha R}\left(PR-\alpha\beta\right)\right),\\
\dot y & =  P-\left(R\left(x+\frac{1}{\alpha R}\left(PR-\alpha\beta\right)\right)+\beta\right)\left(y+\frac{\alpha}{R}\right).
\end{aligned}\right.
\end{equation}
The Jacobian of (\ref{eq:sys_trans}) at $(x,y)=(0,0)$ is given by
\[
J\left(0,0\right)=\left(\begin{array}{cc}
0 & \frac{1}{\alpha}\left(PR-\alpha\beta\right)\\
-\alpha & -\frac{PR}{\alpha}
\end{array}\right).
\]
It has the following eigenvalues:
\[
\lambda_{1}=\frac{1}{2\alpha}\left(-PR-\sqrt{P^{2}R^{2}-4\alpha^{2}\left(PR-\alpha\beta\right)}\right),
\]
and
\[
\lambda_{2}=\frac{1}{2\alpha}\left(-PR+\sqrt{P^{2}R^{2}-4\alpha^{2}\left(PR-\alpha\beta\right)}\right).
\]
In view of these, the equilibrium point is hyperbolic if $PR-\alpha\beta\neq0$,
and the first part of (i) follows. Now we can use the Hartman-Grobman
theorem to characterize this equilibrium point through the linearized
system. Hence, (ii) follows from the requirement that $\lambda_{1}$
and $\lambda_{2}$ must be complex with negative real part, (iii)
is a consequence of $\lambda_{1}$ and $\lambda_{2}$ being negative
real quantities, and (iv) results from $\lambda_{1}$ and $\lambda_{2}$
being real with opposite sign.

We proceed in the same way for $\xi_{2}$. In order to translate this equilibrium
point to the origin, we use the change of variables
\[
x=\rho,\quad 
y=n-\frac{P}{\beta},
\]
in which case, \eqref{eq:ode} becomes
\begin{equation}\label{eq:sys_trans2}
\left\{\begin{aligned}
\dot x & =  \left(R\left(y+\frac{P}{\beta}\right)-\alpha\right)x,\\
\dot y & =  P-\left(Rx+\beta\right)\left(y+\frac{P}{\beta}\right).
\end{aligned}\right.
\end{equation}
The Jacobian of (\ref{eq:sys_trans2}) at $(x,y)=(0,0)$ is given
by
\[
J\left(0,0\right)=\left(\begin{array}{cc}
\frac{PR}{\beta}-\alpha & 0\\
-\frac{PR}{\beta} & -\beta
\end{array}\right),
\]
with eigenvalues:
\[
\lambda_{1}=-\beta,\quad \lambda_{2}=\frac{PR-\alpha\beta}{\beta}.
\]
Therefore, the second part of (i), (v), and (vi) follow as in the
previous cases.
\end{proof}

\begin{remark}
Note that in the case $\alpha = \beta$, the total mass 
of both stationary states $\xi_1$ and $\xi_2$ is given by
\[
M^*=\int_\T \xi_j \, dx \equiv \int_\T \big( \rho^*_j+ n^*_j \big) \, dx = \frac{P |\T|}{\beta}\equiv \frac{P |\T|}{\alpha}, \quad j=1, 2.
\]
which is consistent with Lemma \ref{lem:mas}.
\end{remark}

In the physical relevant case of $\beta \gg \alpha$, the situation with non-vanishing condensate becomes even simpler.

\begin{corollary} \label{cor:ODE}
Let $\alpha>0$ and $\beta>0$ be such that $\beta\gg\alpha$ . Then,
for any values of $P$ and $R$ such that $PR-\alpha\beta>0$, $\xi_{1}$
is an asymptotically stable node. 
\end{corollary}

In particular, $\beta\gg\alpha$
excludes the possibility of $\xi_{1}$ being an asymptotically stable
spiral, and thus we do not expect oscillations of the solution $\psi, n$ near the equilibrium.

\begin{proof}
Assume that we have $\alpha, \beta>0$ with $\beta\gg\alpha$
and we want to find the possible values of $P$ and $R$, with $PR-\alpha\beta>0$,
such that $\xi_{1}$ is either an asymptotically stable spiral or
node. From the results of Theorem \ref{thm:ODE}, we
obtain the inequalities
\[
\begin{array}{l}
\left(PR\right)^{2}-4\alpha^{2}\left(PR\right)+4\alpha^{3}\beta<0\textrm{ (spiral)},\\
\left(PR\right)^{2}-4\alpha^{2}\left(PR\right)+4\alpha^{3}\beta\ge0\textrm{ (node).}
\end{array}
\]
The equation
\[
\left(PR\right)^{2}-4\alpha^{2}\left(PR\right)+4\alpha^{3}\beta=0
\]
has the roots
\[
PR=\left\{ \begin{array}{c}
2\alpha^{2}+\sqrt{4\alpha^{3}\left(\alpha-\beta\right)},\\
2\alpha^{2}-\sqrt{4\alpha^{3}\left(\alpha-\beta\right)}.
\end{array}\right.
\]
Both of these roots are complex if $\beta\gg\alpha$ and one
can verify that for any $P, R$, with $PR-\alpha\beta>0$,
and $\beta\gg\alpha$, the only possibility is 
\[
\left(PR\right)^{2}-4\alpha^{2}\left(PR\right)+4\alpha^{3}\beta>0.
\]
Notice that this inequality is also valid for $PR=0$. Hence,
$\beta\gg\alpha$ ensures that $\xi_{1}$ is an asymptotically stable
node. 
\end{proof}

Figures (\ref{fig:spiral}) and (\ref{fig:node}) below show the phase portrait
of (\ref{eq:ode}) for different values of the parameters. The
numerical simulations have been obtained using a standard fourth-order Runge-Kutta method, and 
agree with the results of Theorem \ref{thm:ODE}.

\begin{figure}[H]
\begin{centering}
\includegraphics[scale=0.6]{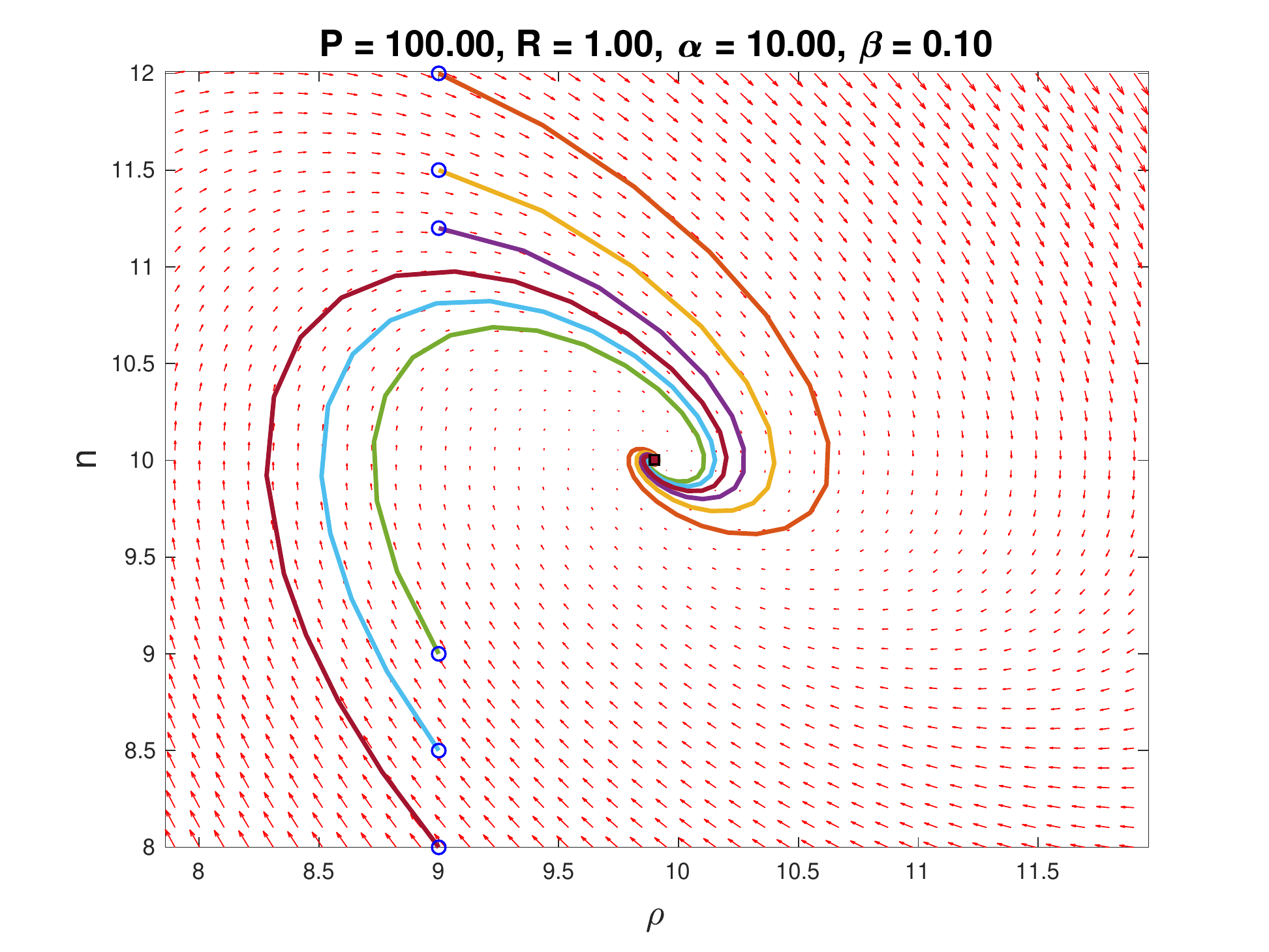}
\par\end{centering}
\caption{Asymptotically stable spiral $\xi_{1}$, with
$0<\left(\frac{PR}{2\alpha}\right)^2<PR-\alpha\beta$.}\label{fig:spiral}
\end{figure}
\begin{figure}[H]
\begin{centering}
\includegraphics[scale=0.6]{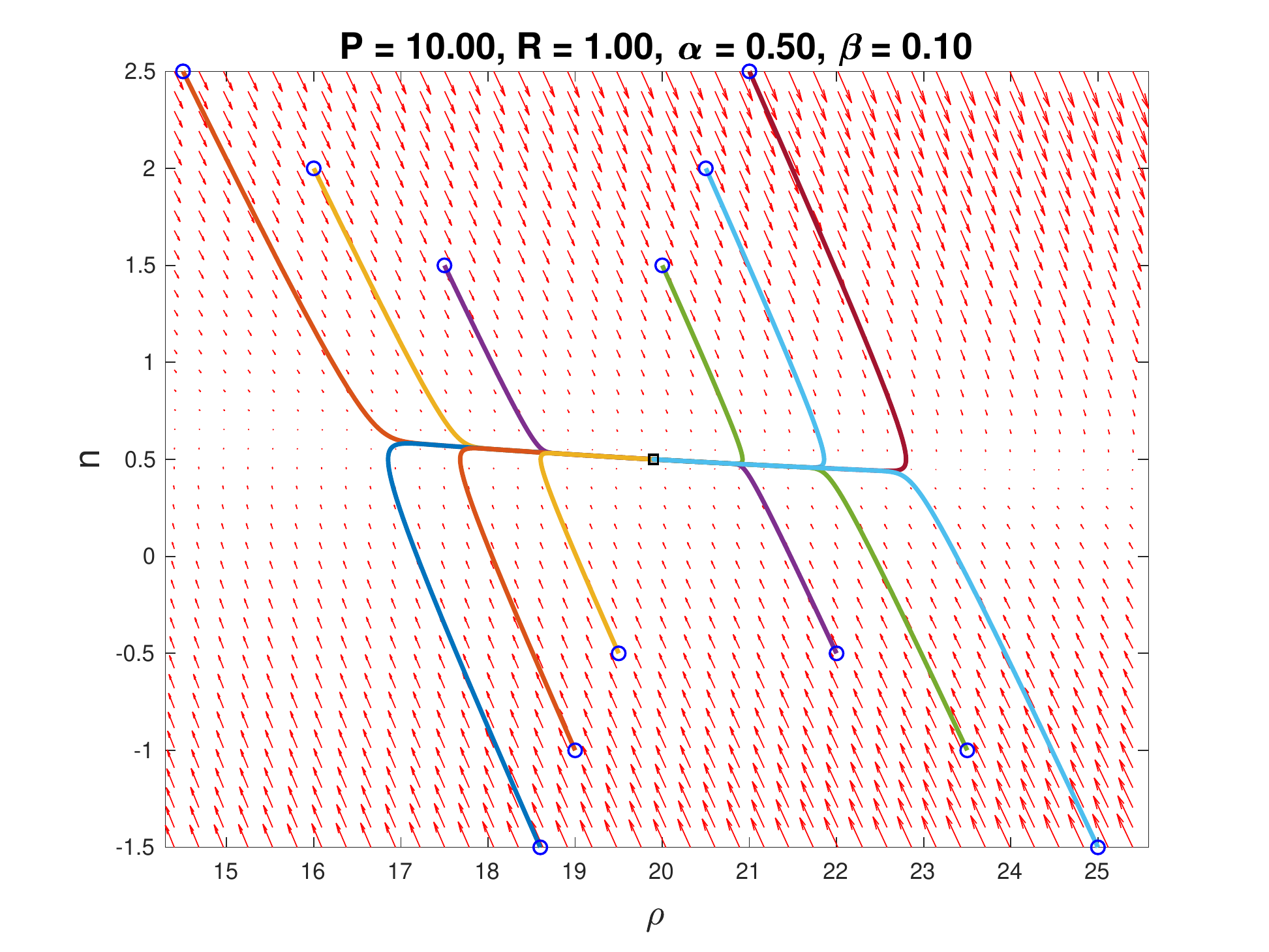}
\par\end{centering}
\caption{Asymptotically stable node $\xi_{1}$, with $0<PR-\alpha\beta\le\left(\frac{PR}{2\alpha}\right)^2.$ }\label{fig:node}
\end{figure}

As we have seen, both $\xi_1, \xi_2$ are hyperbolic, except if $PR-\alpha \beta=0$. 
Determining the stability and qualitative behavior of a dynamical system in 
a neighborhood of a non-hyperbolic critical point requires a
different approach, such as the center manifold theory. However, we
will not discuss this situation since our primary concern is $PR-\alpha\beta>0$. 
We shall only add that in the case $PR-\alpha\beta=0$ the system
\eqref{eq:ode} has a single non-hyperbolic critical point given
by
\[
\xi=\left(\begin{array}{c}
\rho^{*}\\
n^{*}
\end{array}\right)=\left(\begin{array}{c}
0\\
\frac{\alpha}{R}
\end{array}\right)\equiv\left(\begin{array}{c}
0\\
\frac{P}{\beta}
\end{array}\right).
\]
Moreover, in our numerical simulations, $\xi$ behaves like a node when approached from $\rho>0$ and like a saddle point when approached
from $\rho<0$. This behavior is commonly observed in non-hyperbolic
equilibrium points (see \cite{Per}).

\begin{remark}\label{rem:abel}
Note that \eqref{eq:ode} can be reduced to the following first
order equation and quadrature:
\begin{equation}
\frac{dn}{d\rho}=\frac{P-\left(R\rho+\beta\right)n\left(\rho\right)}{\left(Rn\left(\rho\right)-\alpha\right)\rho},\label{eq:Abel1}
\end{equation}
together with
\begin{equation*}
t=\int\frac{d\rho}{\left(Rn\left(\rho\right)-\alpha\right)\rho}+c,\label{eq:quad_Abel}
\end{equation*}
where $c$ is an integration constant. Equation \eqref{eq:Abel1} is an Abel equation of the second kind, which is a well-studied class of equations, see, e.g. \cite{ZP}.
Unfortunately, \eqref{eq:Abel1} does not seem to fit any of the explicitly solvable examples currently known. We have to consider this fact later for our numerical scheme. 
\end{remark}


\subsection{A Lyapunov functional for $\xi_2$}

Recall that the equilibrium point $\xi_2$ defined in \eqref{eq:equi2}, describing the situation with vanishing condensate, is asymptotically stable if 
$PR-\alpha\beta<0$. Under this condition, it is possible to define a Lyapunov functional for the ODE system \eqref{eq:ode}. 
To this end, we first note that \eqref{eq:ode} can be rewritten as
\begin{equation}\label{eq:ode1}
\left\{\begin{aligned}
\dot\rho=&\frac1\beta(PR-\alpha\beta)\rho+R\left(n-\frac{P}{\beta}\right)\rho\\
\dot n=& -(R\rho+\beta)\left(n-\frac{P}{\beta}\right)-\frac{PR}{\beta}\rho.
\end{aligned}\right.
\end{equation}
In this way, it is easy to see that the following holds:

\begin{lemma} The expression
\begin{equation*}
\ell (t):=\frac{P}{\beta}\rho(t)+\frac12\left(n(t)-\frac{P}{\beta}\right)^2,
\end{equation*}
is a Lyapunov functional for \eqref{eq:ode1}, provided $PR-\alpha\beta<0$ and $\rho\ge0$. In particular, we 
have that
\[
\left(\begin{array}{c}
\rho(t)\\
n(t)
\end{array}\right) \xrightarrow{t \to +\infty} \xi_{2}=\left(\begin{array}{c}
0\\
\frac{P}{\beta}
\end{array}\right), \quad \text{exponentially fast.}
\]
\end{lemma}
\begin{proof} 
Using \eqref{eq:ode1} we simply compute the time-derivative of $\mathcal L$:
\begin{equation*}
\begin{aligned}
&\frac{d\ell}{dt} =\\
&=\frac{P}{\beta^2}(PR-\alpha\beta)\rho+\frac{PR}{\beta}\rho\left(n-\frac{P}{\beta}\right)-(R\rho+\beta)\left(n-\frac{P}{\beta}\right)^2-\frac{PR}{\beta}\rho\left(n-\frac{P}{\beta}\right)\\
& \le \frac{P}{\beta^2}(PR-\alpha\beta)\rho-\beta\left(n-\frac{P}{\beta}\right)^2\\
& \le -c \ell,
\end{aligned}
\end{equation*}
for some $c>0$. Thus 
\[
\ell(t) \le e^{-ct}  \ell(0), \quad \text{for all $t\ge 0$,} 
\]
which directly implies exponential decay of $\rho$ and $n$.
\end{proof}

\begin{remark} This simple idea can even be lifted to the level of the original PDE-system \eqref{GP}. Indeed, let
\[
{L}(t):= \frac{P}{\beta}\int_\T |\psi(t,x)|^2 \, dx + \frac12\int_\T \left(n(t,x)-\frac{P}{\beta}\right)^2 \, dx.
\]
Differentiating ${ L}$ with respect to time and using the first equation from \eqref{QHD}, yields 
an exponentially fast decay in-time of ${ L}$, along the same lines as before. Assuming $T_{\rm max}=+\infty$, this clearly implies 
that, as $t\to +\infty$: $\psi(t,x) \to 0$, and $n(t,x)\to \frac{P}{\beta}$, exponentially fast. 
\end{remark}


\section{Numerical simulations}\label{sec:num}

In this section, we study the (long-time) behavior of solutions of \eqref{GP} with general (non-space-homogeneous) initial data via numerical integration. In particular, we are interested in the evolution of the system after perturbing the space-homogeneous solutions obtained in Section \ref{sec:space}. This approach will give us an insight into the attractor of the PDE system \eqref{GP} and a way to compare it with that of the ODE system \eqref{eq:ode}.

\subsection{Stationary states} Before presenting the details of our numerical computations, we shall briefly comment on some basic properties of general $x$-dependent steady states. These are solutions of \eqref{QHD} given by
\begin{equation}\label{steady}
\psi(t,x) = e^{-i \mu t} \varphi(x), \quad n(x) = \frac{P}{R |\varphi|^2 + \beta},
\end{equation}
where $\mu \in \R$ and $\varphi(x)\in \C$, some yet undetermined wave function, which is only unique up to a constant phase factor.

\begin{lemma}
A necessary condition for the existence of non-trivial steady states $0\not = \varphi \in H^1(\T)$, and hence $ n = \frac{P}{R |\varphi|^2 + \beta} \in L^\infty(\T)$, is: 
\[
PR-\alpha\beta>0 , \quad \text{and} \quad \mu >0.
\]
\end{lemma}

\begin{proof}
Plugging the ansatz \eqref{steady} into \eqref{QHD} yields the following equation for $\varphi$:
\[
\mu \varphi=-\frac12\partial_x^2\varphi+g|\varphi|^2\varphi+ \frac{\lambda P \varphi}{\beta+ R|\varphi |^2} +\frac{i}{2}\left(\frac{PR}{\beta+ R|\varphi|^2}-\alpha\right)\varphi.
\]
Here $\mu$ plays the role of a chemical potential. 
Multiplying this equation by $\overline \varphi$ and separating real and imaginary parts, we find, after some straightforward computations,
\begin{equation}\label{steadyGP}
\left\{\begin{aligned}
&\mu |\varphi|^2=-\frac14\partial_x^2|\varphi|^2+\frac{1}{2} |\partial_x \varphi|^2 +g|\varphi|^4+ \frac{\lambda P |\varphi|^2}{\beta+ R|\varphi |^2}, \\
&0 =-\frac{1}{2}\partial_x \im(\overline \varphi \partial_x \varphi )+ \left(\frac{PR}{\beta+ R|\varphi|^2}-\alpha\right)|\varphi|^2,
\end{aligned}\right.
\end{equation}
By integrating the second equation over $\mathbb T$, the term involving the imaginary part vanishes and we thus have
\[
(PR - \alpha \beta ) \int_\T \frac{|\varphi|^2}{R|\varphi|^2+\beta}\, dx = \alpha R \int_\T \frac{|\varphi|^4}{R|\varphi|^2+\beta}\, dx.
\]
This implies $PR-\alpha\beta>0$ for otherwise $\varphi \equiv 0$. Also, by integrating the first equation of \eqref{steadyGP} over $\T$, we obtain
\[
\mu \int_\T |\varphi|^2\, dx = \frac{1}{2} \int_\T  |\partial_x \varphi|^2 +g\int_\T |\varphi|^4\, dx + \lambda P \int_\T \frac{ |\varphi|^2}{\beta+ R|\varphi |^2} \, dx,
\]
which clearly implies $\mu>0$, since $g, \lambda, P >0$ by assumption.
\end{proof}

Note that the second equation of \eqref{steadyGP} also shows that any {\it real-valued} (up to a constant phase) steady state wave function $\varphi\not =0$ is necessarily equal to 
\begin{equation}\label{stationary}
|\varphi|^2 = \frac{PR-\alpha \beta}{\alpha R},
\end{equation}
i.e., the same constant as that obtained in Theorem \ref{thm:ODE}. At the moment, we cannot exclude the possibility of 
complex steady states, $\varphi$, not obtained from a real function by a constant rotation of phase. On the other hand, we have not seen this situation in our numerical simulations. Such $\varphi\in \C$ would 
correspond to {\it non-equilibrium steady states} with non-vanishing current density, $J=\im(\overline \varphi \partial_x \varphi )\not =0$.


\subsection{Numerical method} Below, we shall present several numerical findings for solutions of our model system (\ref{GP}) with general (non-space-homogeneous) initial data. 
These numerical results are obtained using a Strang-splitting Fourier spectral method. 

Let $h=\Delta x>0$ denote the mesh size, with $h=\left|\mathbb{T}\right|/M$, where
$M\in 2\N$. Define $\tau=\Delta t>0$ to be the
time-step. Let the grid points be $x_{j}=a+jh,\: j=0,1,...,M$, and $t_{n}=n\tau,\: n=0,1,2,...$. The main idea is to split 
the system (\ref{GP}) into:
\begin{equation}\label{eq:sp_1}
\left\{\begin{aligned}
&i\partial_t\psi=  g\left|\psi\right|^{2}\psi+\lambda n\psi+\frac{i}{2}\left(Rn-\alpha\right)\psi,\\
& \partial_t n=  P-\left(R\left|\psi\right|^{2}+\beta\right)n
\end{aligned}\right.
\end{equation}
and
\begin{equation}\label{eq:sp_2}
i\partial_t\psi=-\frac12\partial_x^2 \psi.
\end{equation}

Notice that \eqref{eq:sp_1} is an ODE system. It is important to remark that this splitting method is particularly useful when the corresponding ODE system can be explicitly integrated. In such cases, one can usually show that the method is unconditionally stable, among other properties (see, e.g., \cite{Ba,Ba2,Si} ). To deal with the ODE resulting from the splitting method, one usually considers the WKB ansatz $\psi=\sqrt{\rho}e^{i\phi}$. If we proceed in this way for the system \eqref{eq:sp_1}, we end up with the system \eqref{eq:ode}. As indicated in Remark \ref{rem:abel}, a similarity reduction of \eqref{eq:ode} leads to an Abel equation of the second kind with no explicit solution. 

Since we are not able to explicitly integrate the ODE system \eqref{eq:sp_1}, we have to rely on numerical integration. In particular, the stability of the method used to integrate the ODE system will determine the stability of the entire numerical scheme.

On the other hand, we discretize \eqref{eq:sp_2} in space by a Fourier spectral method 
and then integrate in-time exactly via $$\psi = \mathcal F^{-1}\left(e^{it \xi^2/2} \mathcal F(\psi_0)\right).$$ 

Combining these two steps using a Strang-splitting yields a numerical solution, $\Psi_n^j \approx \psi(t_n, x_j)$, on the time-interval $\left[t_{n},t_{n+1}\right]$. 

When choosing an integration method for the ODE system \eqref{eq:sp_1}, we have to consider that the splitting method will be at most second-order accurate. Besides, it is essential to keep in mind the stability of the numerical scheme, as mentioned before. 

To corroborate the results presented below, we have used two different methods for the numerical integration of \eqref{eq:sp_1}: a fourth-order Runge-Kutta method and a second-order midpoint method.


\subsection{Numerical results} In this section, our primary goal is to study the time-evolution of certain 
perturbations of the space-homogenous solutions depicted in Figs. \ref{fig:spiral} and \ref{fig:node}.

Fig. \ref{fig:Test1} shows the time evolution of the position density
of the perturbed stationary solution corresponding to $\alpha=10$,
$\beta=0.1$, $R=1$, and $P=100$. In particular $PR-\alpha \beta>0$ in this case. Notice that after a transient phase,
the system returns to the stationary (space-homogeneous) solution \eqref{stationary}.

\begin{figure}[H]
\begin{centering}
\begin{tabular}{cc}
\includegraphics[scale=0.32]{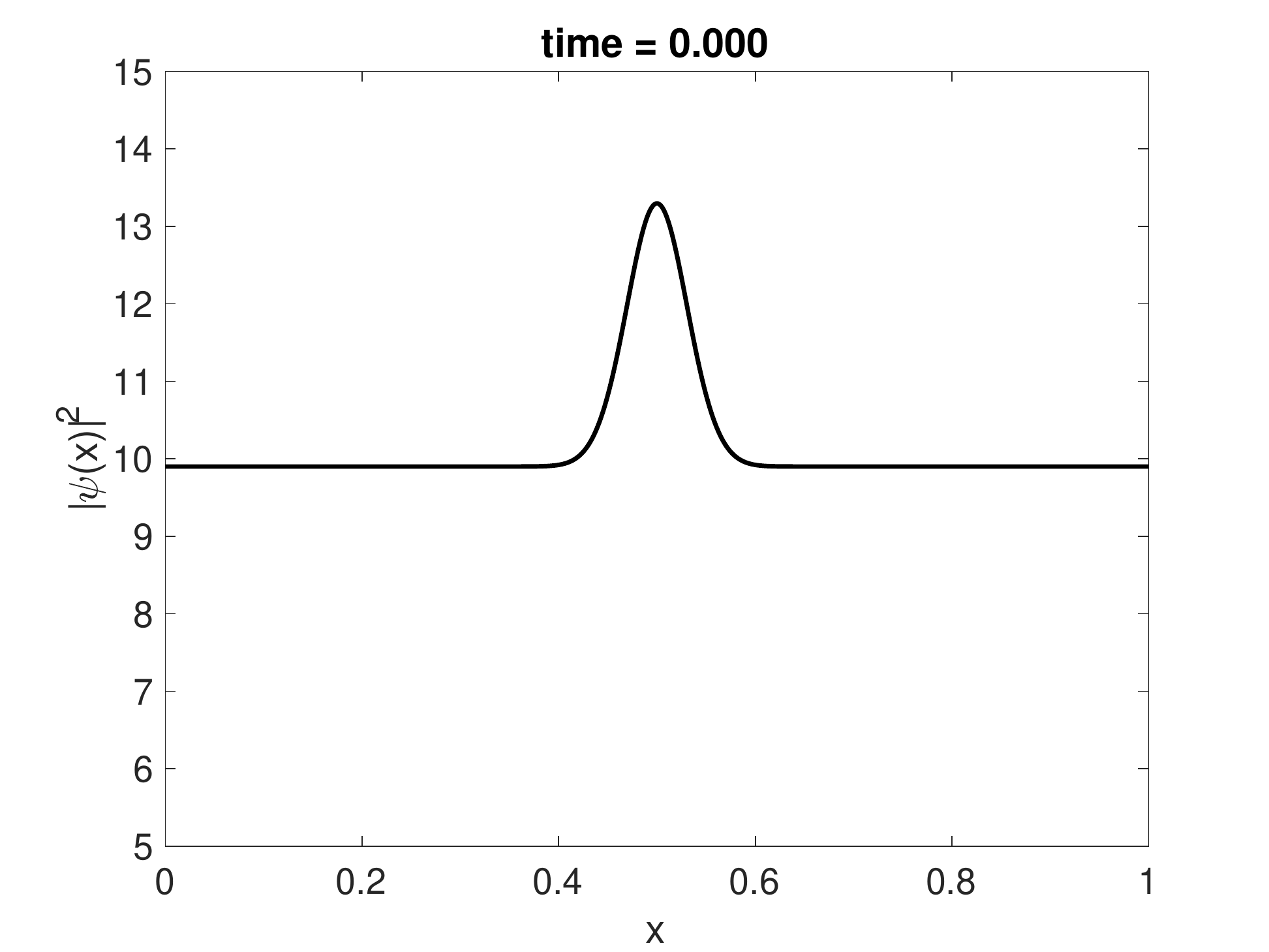} & \includegraphics[scale=0.32]{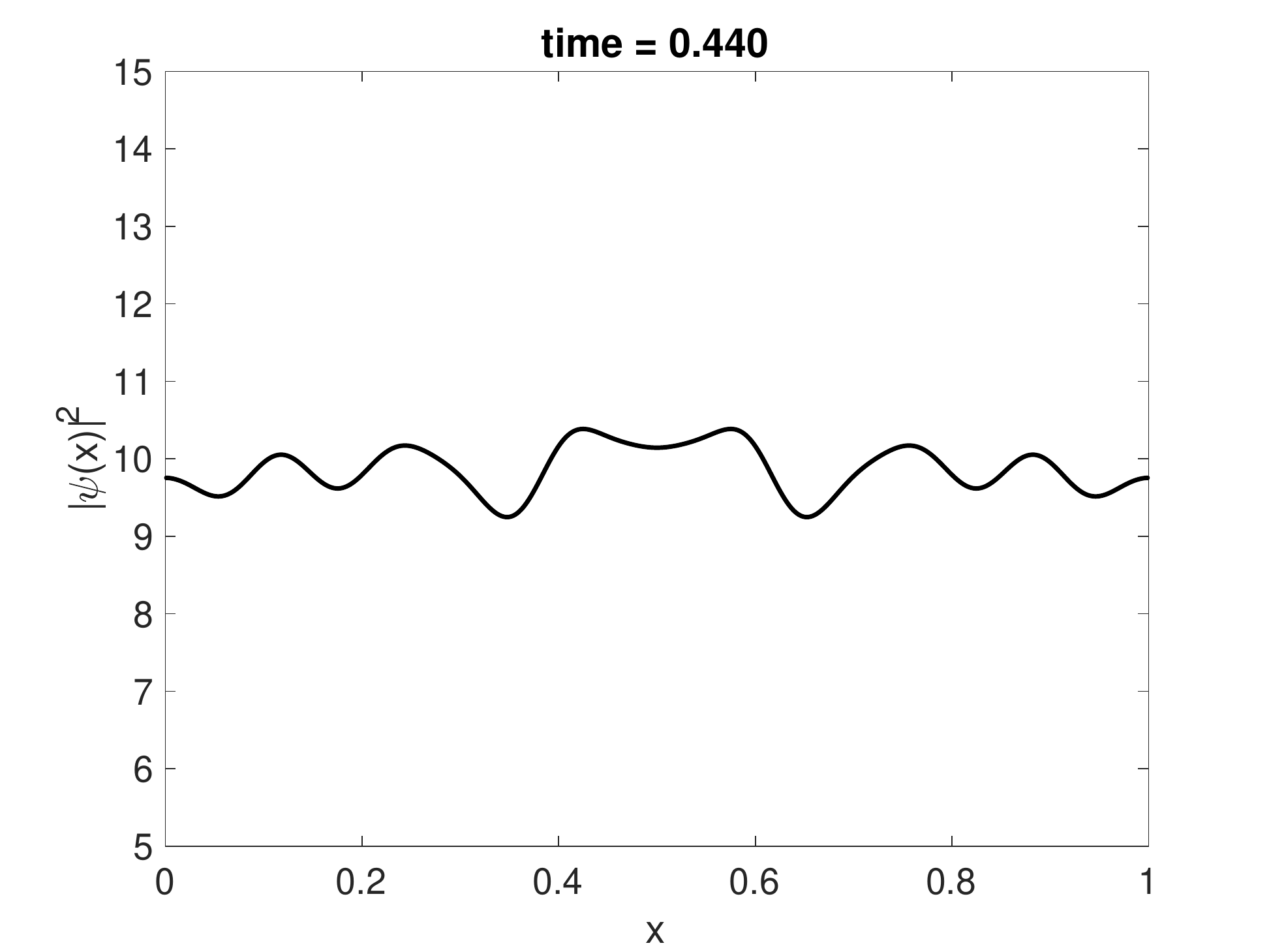}\tabularnewline
\includegraphics[scale=0.32]{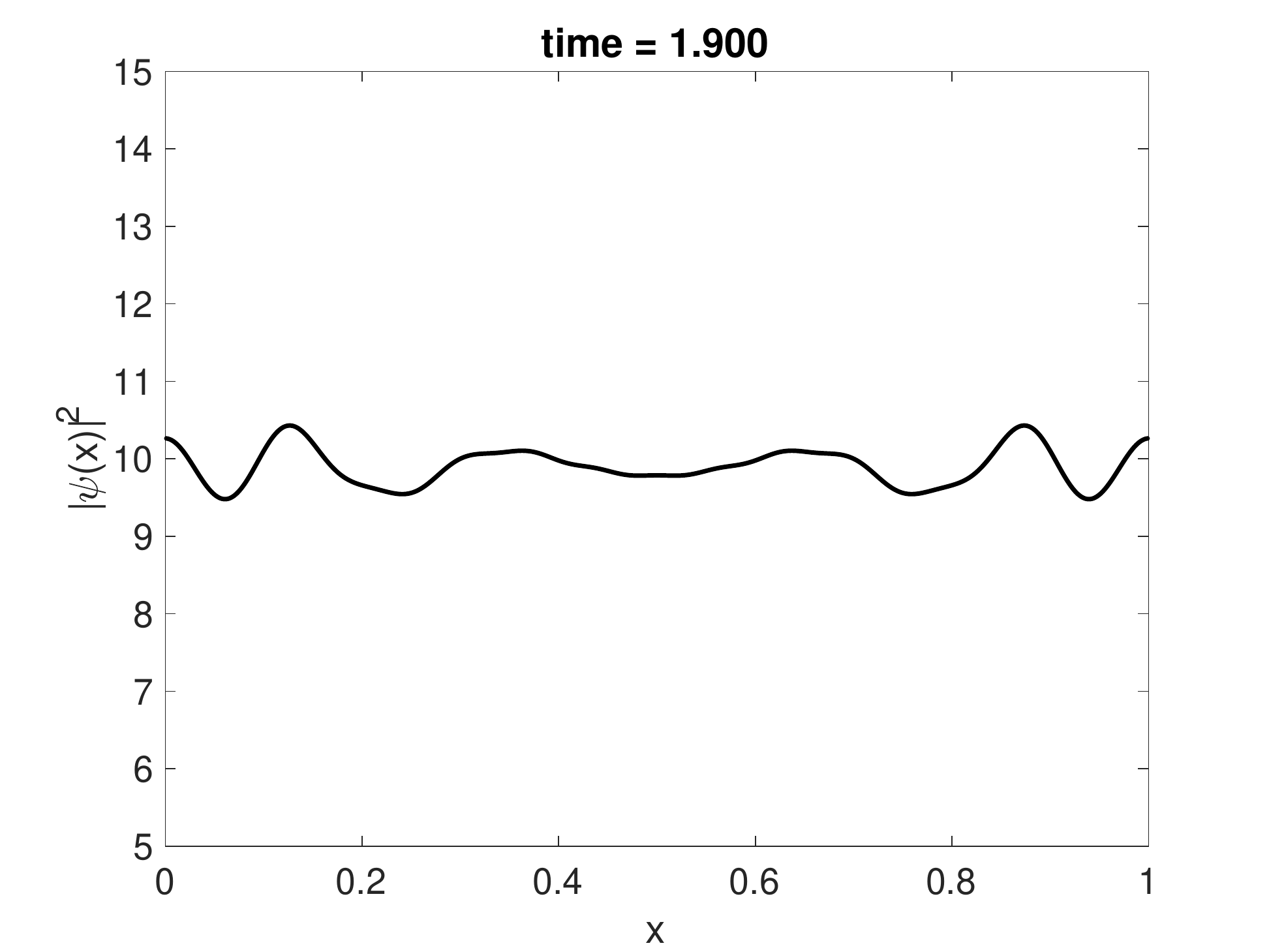} & \includegraphics[scale=0.32]{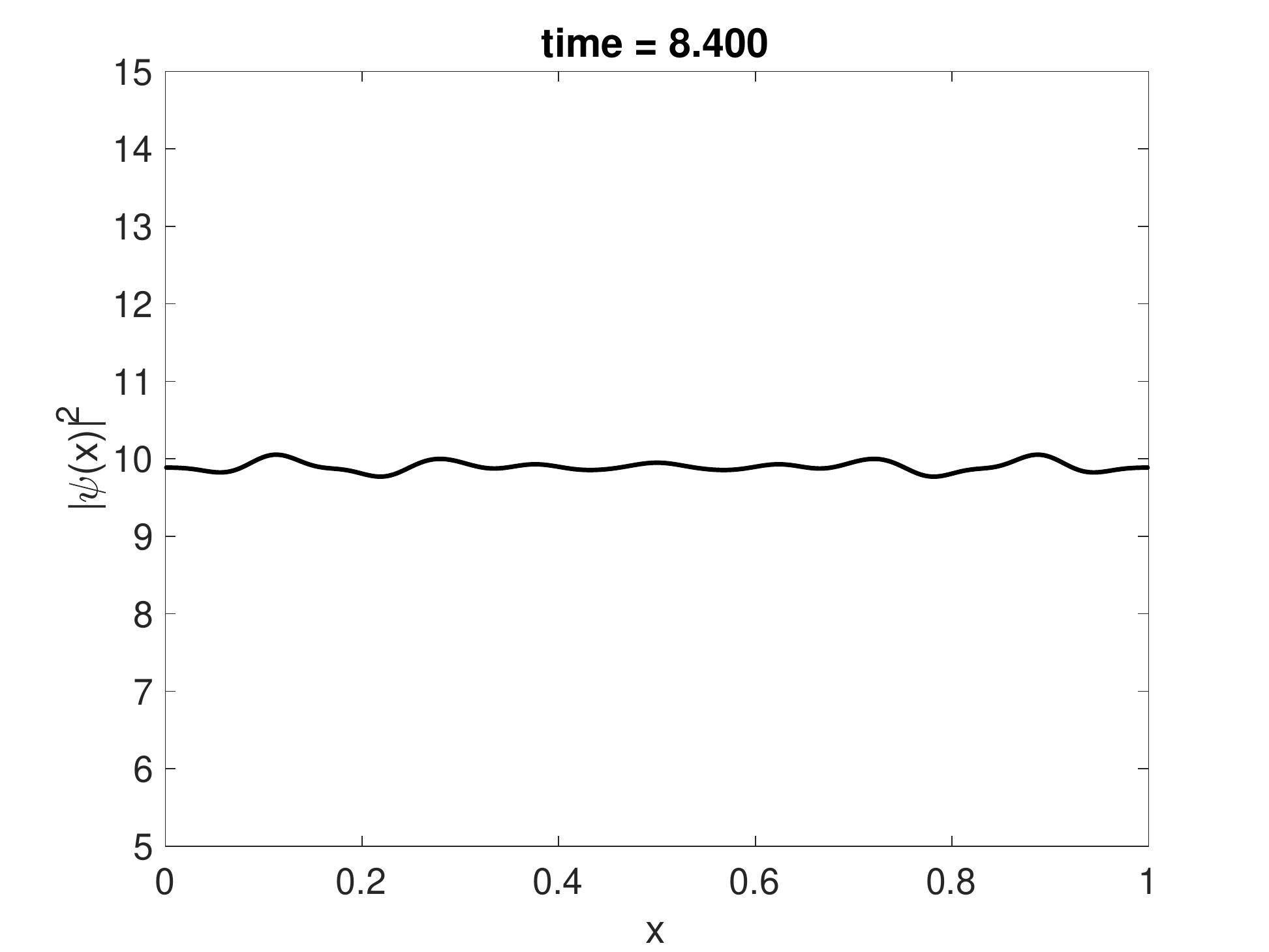}\tabularnewline
\includegraphics[scale=0.32]{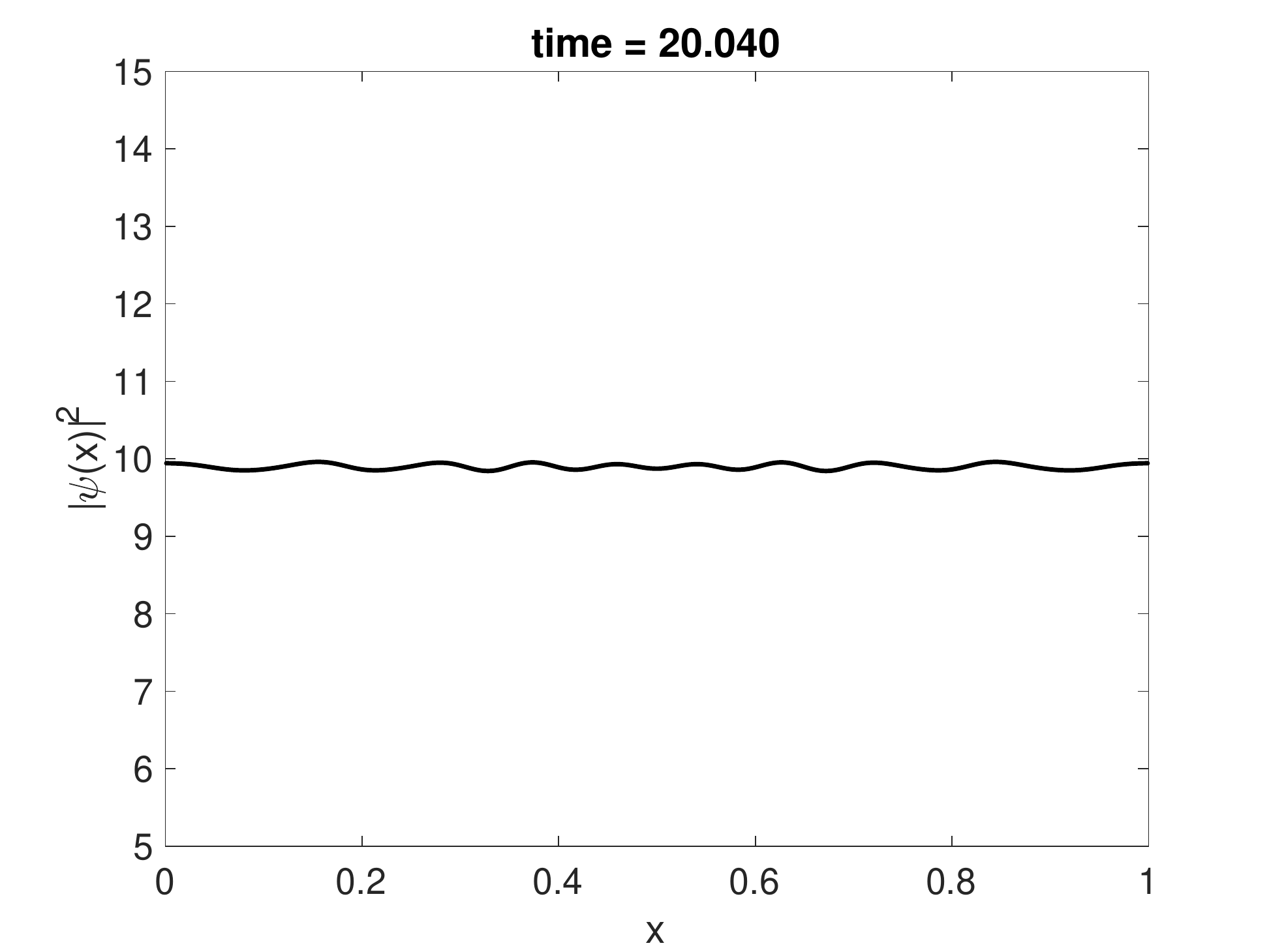} & \includegraphics[scale=0.32]{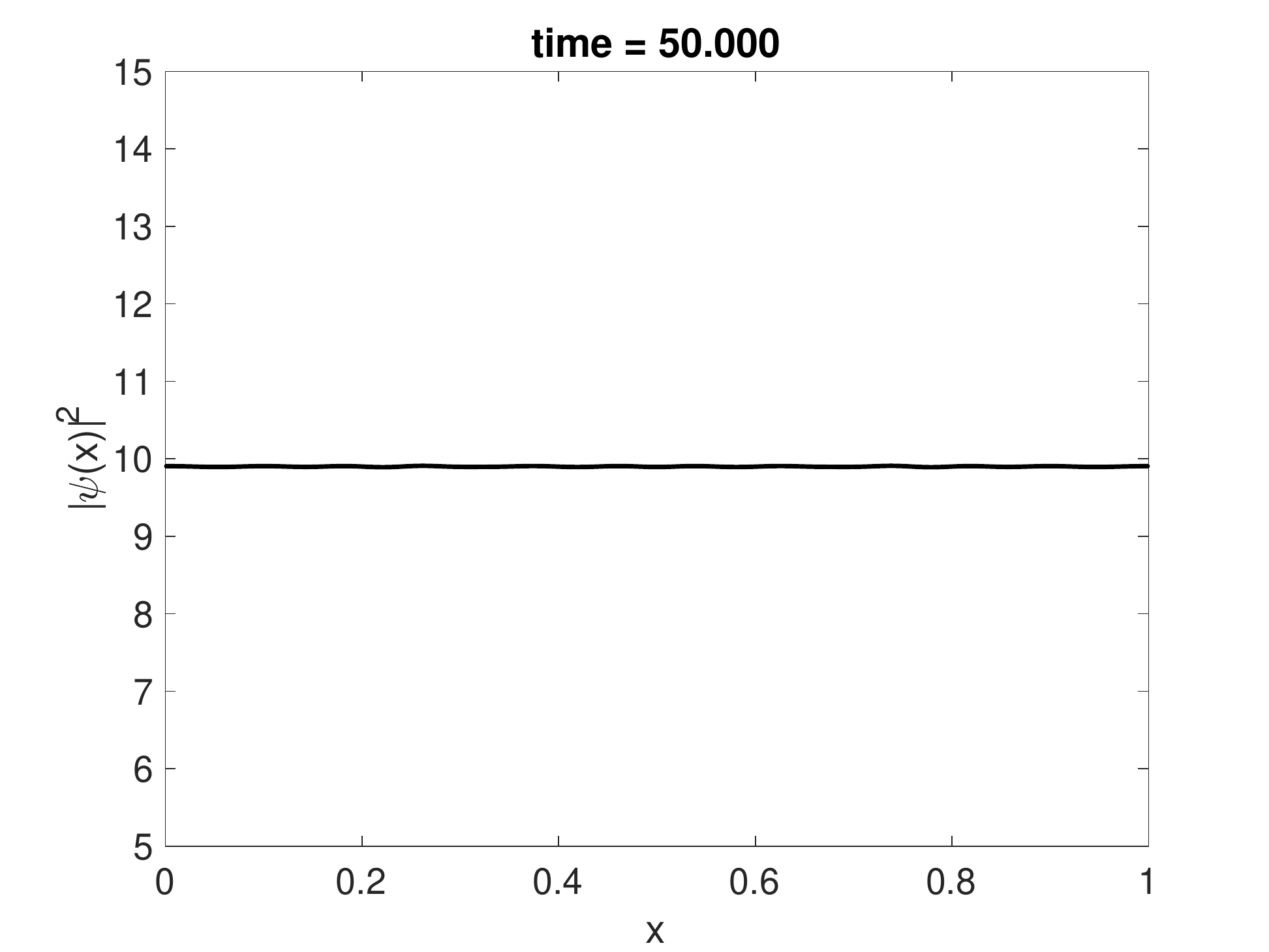}\tabularnewline
\end{tabular}
\par\end{centering}
\caption{Time evolution of the position density of the perturbed stationary
solution corresponding to $\alpha=10$, $\beta=0.1$, $R=1$, and
$P=100$. }\label{fig:Test1}
\end{figure}

Figure \ref{fig:L2_L1_exp1} shows the evolution of (the square of)
the $L^{2}-$norm of $\psi$ and the $L^{1}-$norm of $n$ corresponding to the simulation displayed in Fig. \ref{fig:Test1}.
Furthermore, Fig. \ref{fig:inf_phase} shows the plot of the $L^{2}-$norm
of $\psi$ vs the $L^{1}-$norm of $n$ corresponding to: (0), the simulation
shown in Fig. \ref{fig:Test1}; (1) and (2), the simulations with 
initial conditions shown in Fig. \ref{fig:IC_spiral}. Notice the similarities with the case of space-homogeneous solutions studied in Section \ref{sec:space}; in particular, Fig.
\ref{fig:inf_phase} resembles an asymptotically stable spiral, which should be compared with Fig. \ref{fig:spiral}.

\begin{figure}[H]
\begin{centering}
\begin{tabular}{cc}
\includegraphics[scale=0.32]{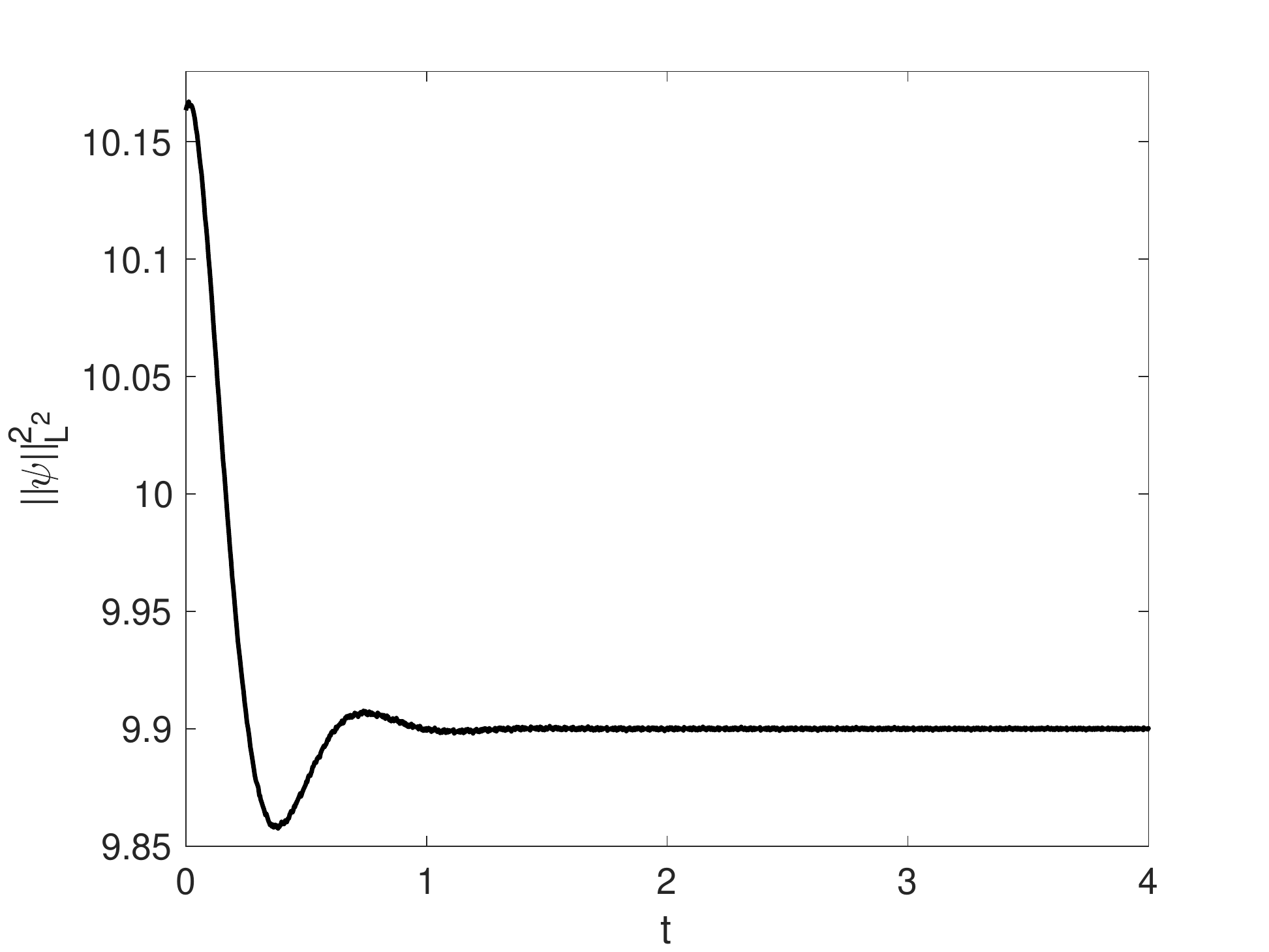} & \includegraphics[scale=0.32]{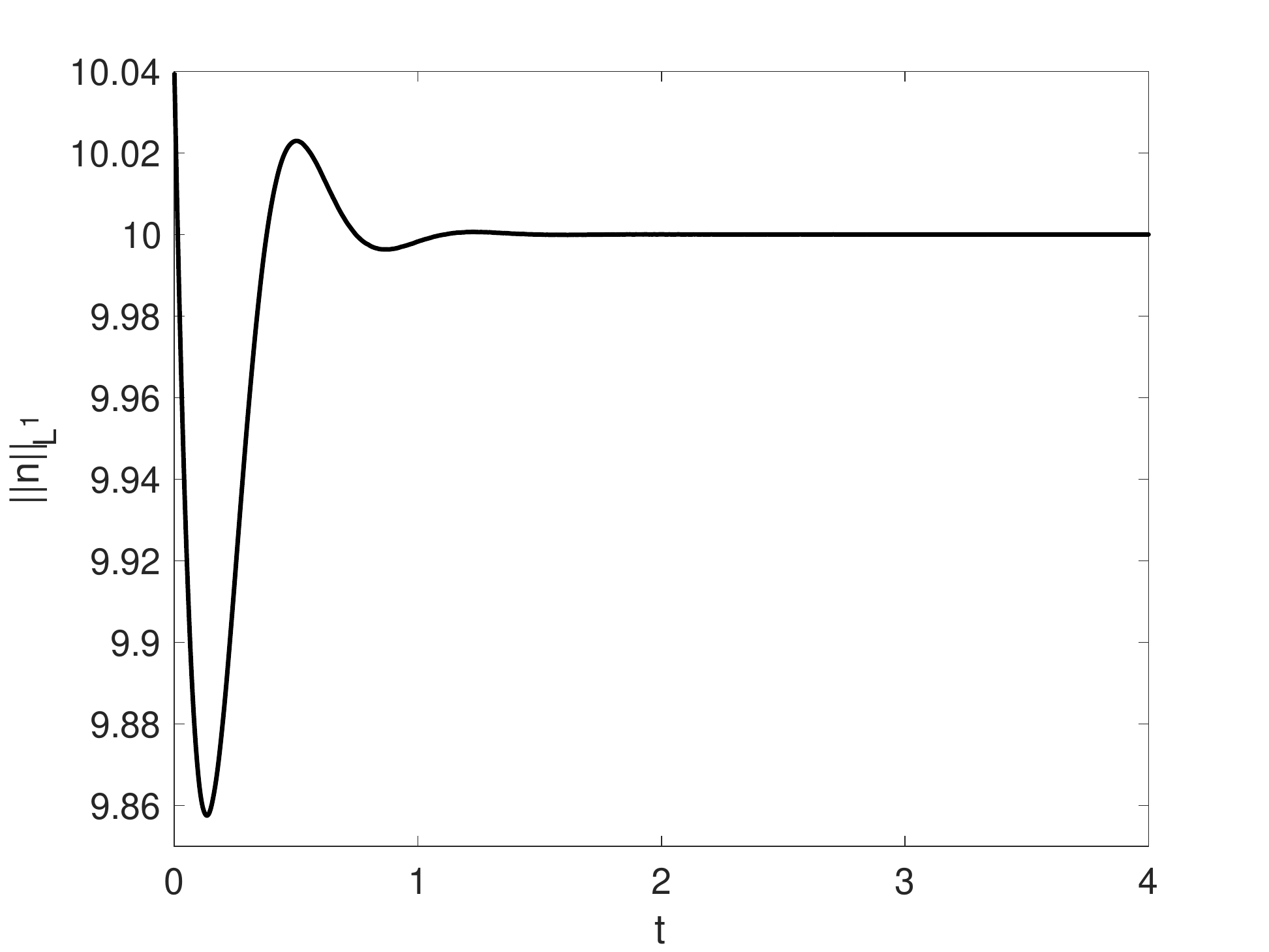}\tabularnewline
\end{tabular}
\par\end{centering}
\caption{Evolution of the $L^{2}-$norm of $\psi$ (left) and the $L^{1}-$norm
of $n$ (right) corresponding to the simulation shown in Fig. (\ref{fig:Test1}).
}\label{fig:L2_L1_exp1}
\end{figure}

\begin{figure}[H]
\begin{centering}
\includegraphics[scale=0.7]{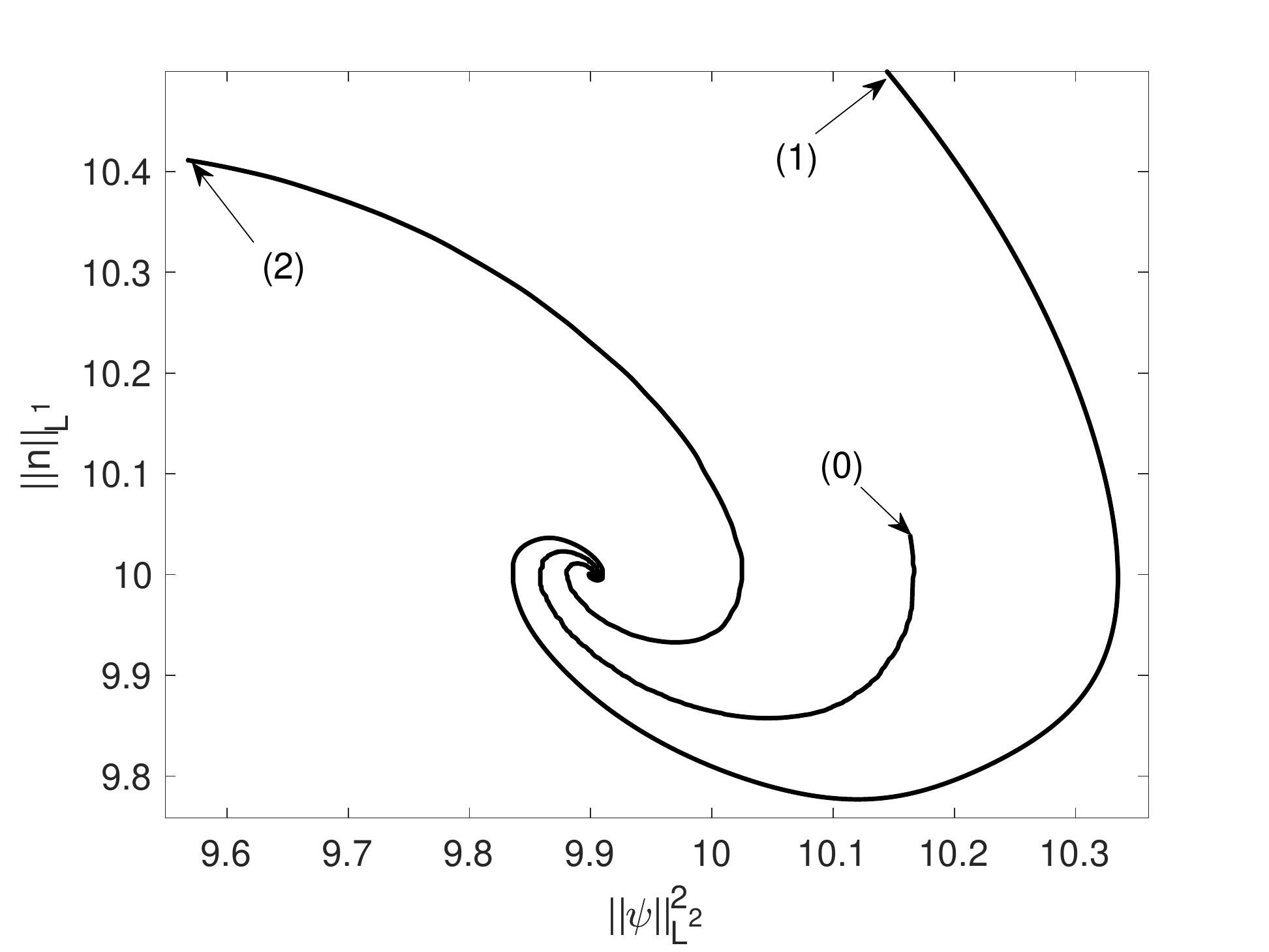}
\par\end{centering}
\caption{$L^{2}-$norm of $\psi$ vs $L^{1}-$norm of $n$ corresponding to: (0), the simulation
shown in Fig. \ref{fig:Test1}; (1) and (2), the simulations with
the initial conditions shown in Fig. \ref{fig:IC_spiral}.}\label{fig:inf_phase}
\end{figure}

\begin{figure}[H]
\begin{centering}
\begin{tabular}{cc}
\includegraphics[scale=0.32]{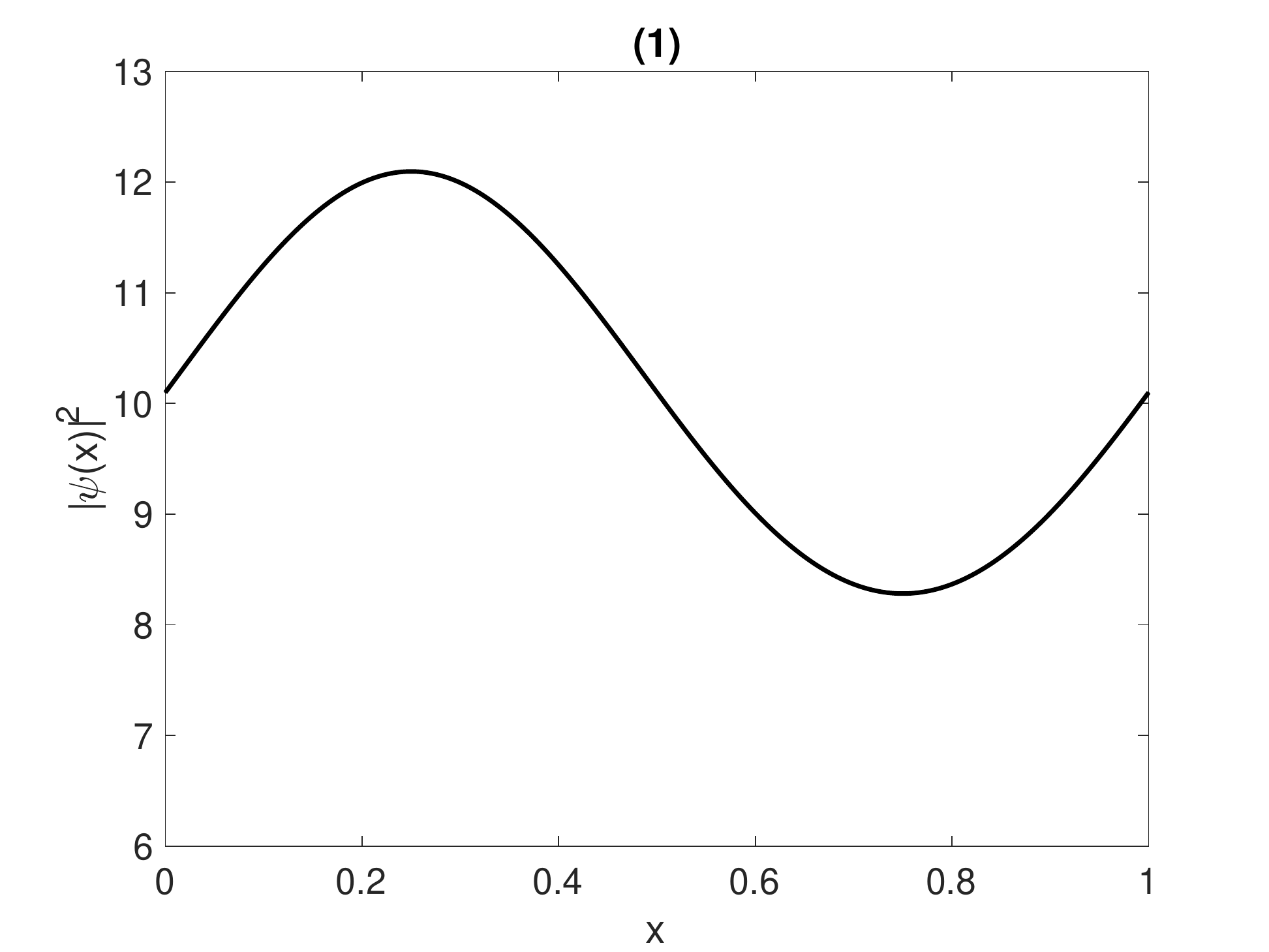} & \includegraphics[scale=0.32]{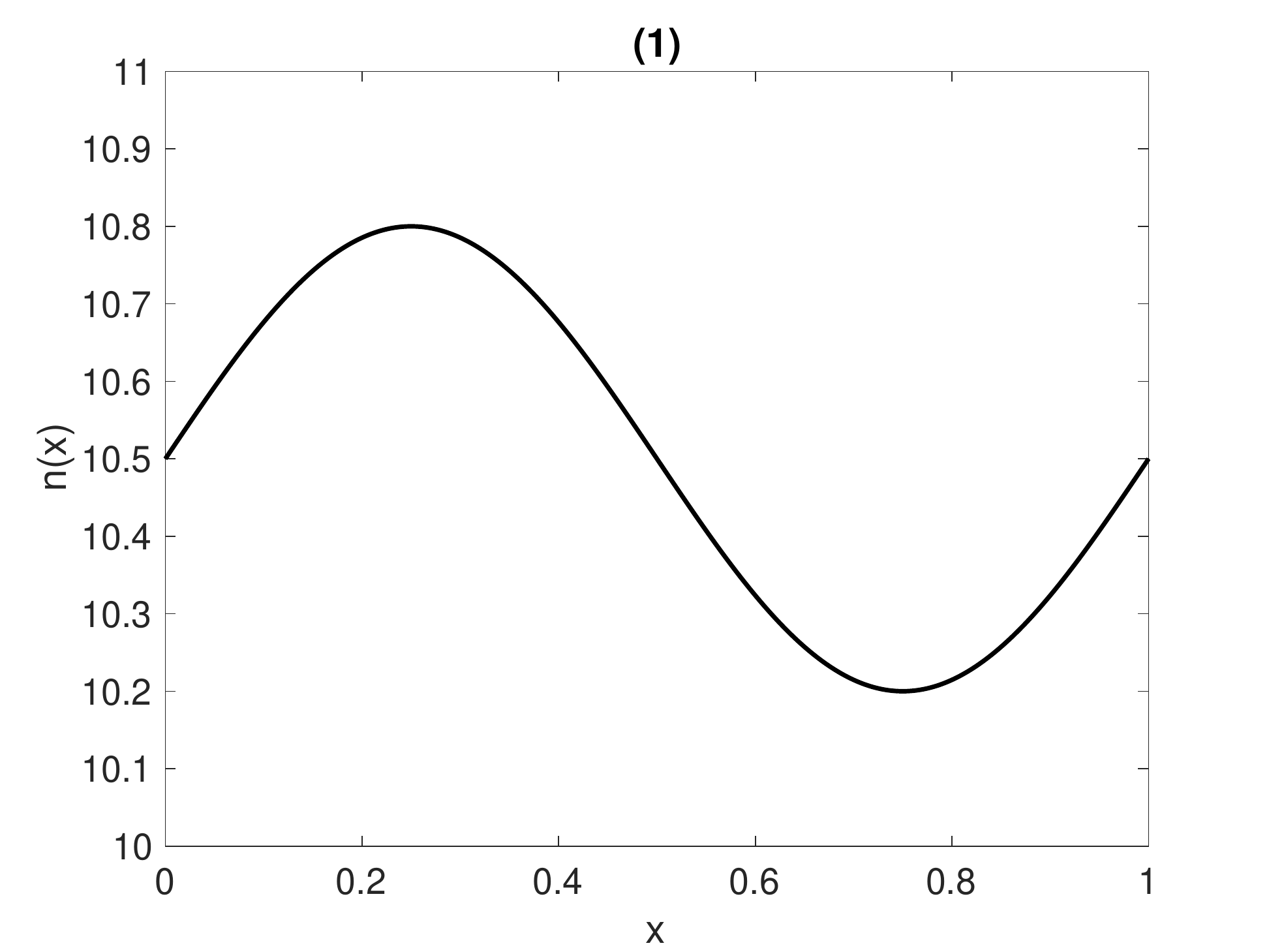}\tabularnewline
\includegraphics[scale=0.32]{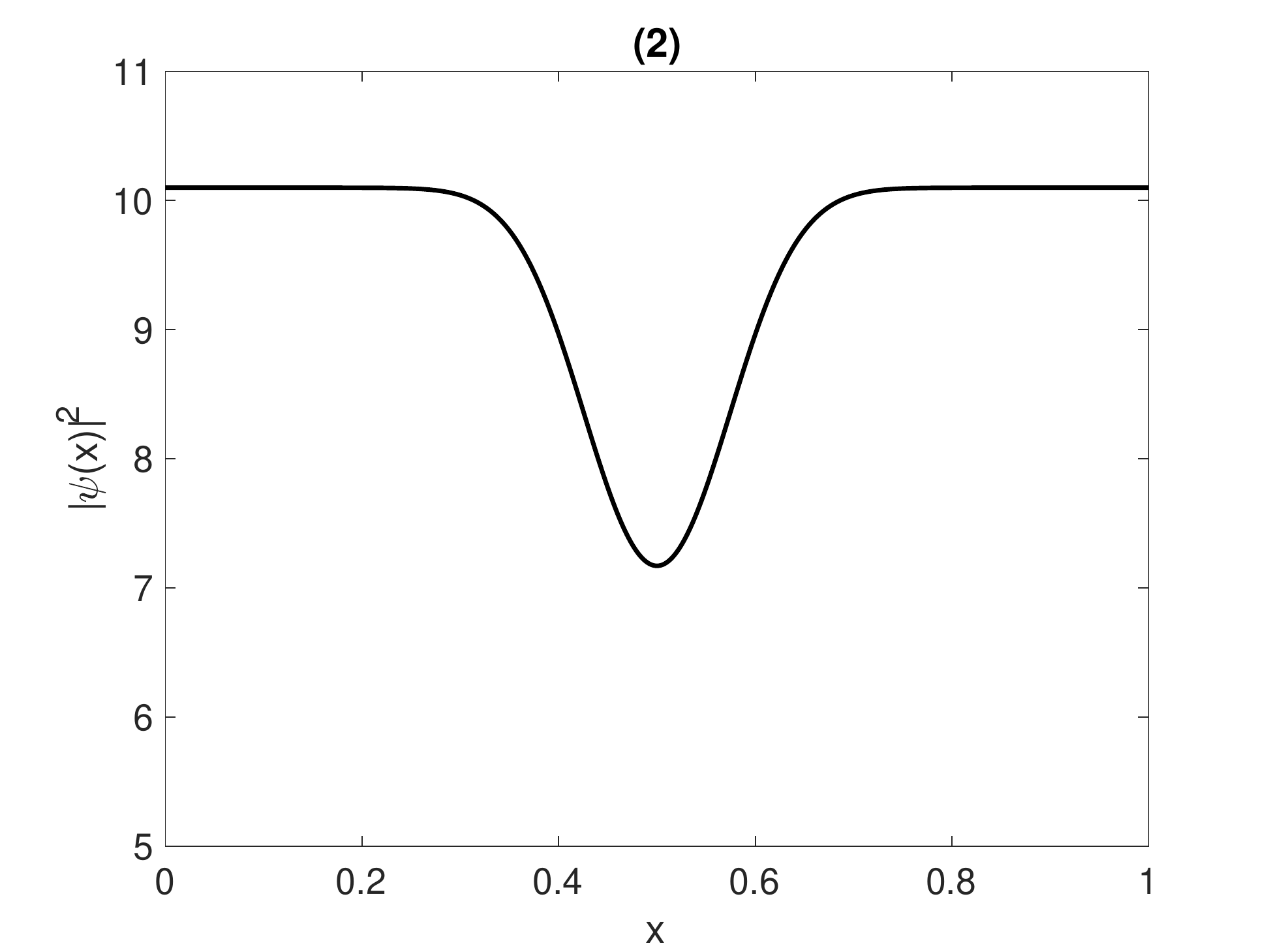} & \includegraphics[scale=0.32]{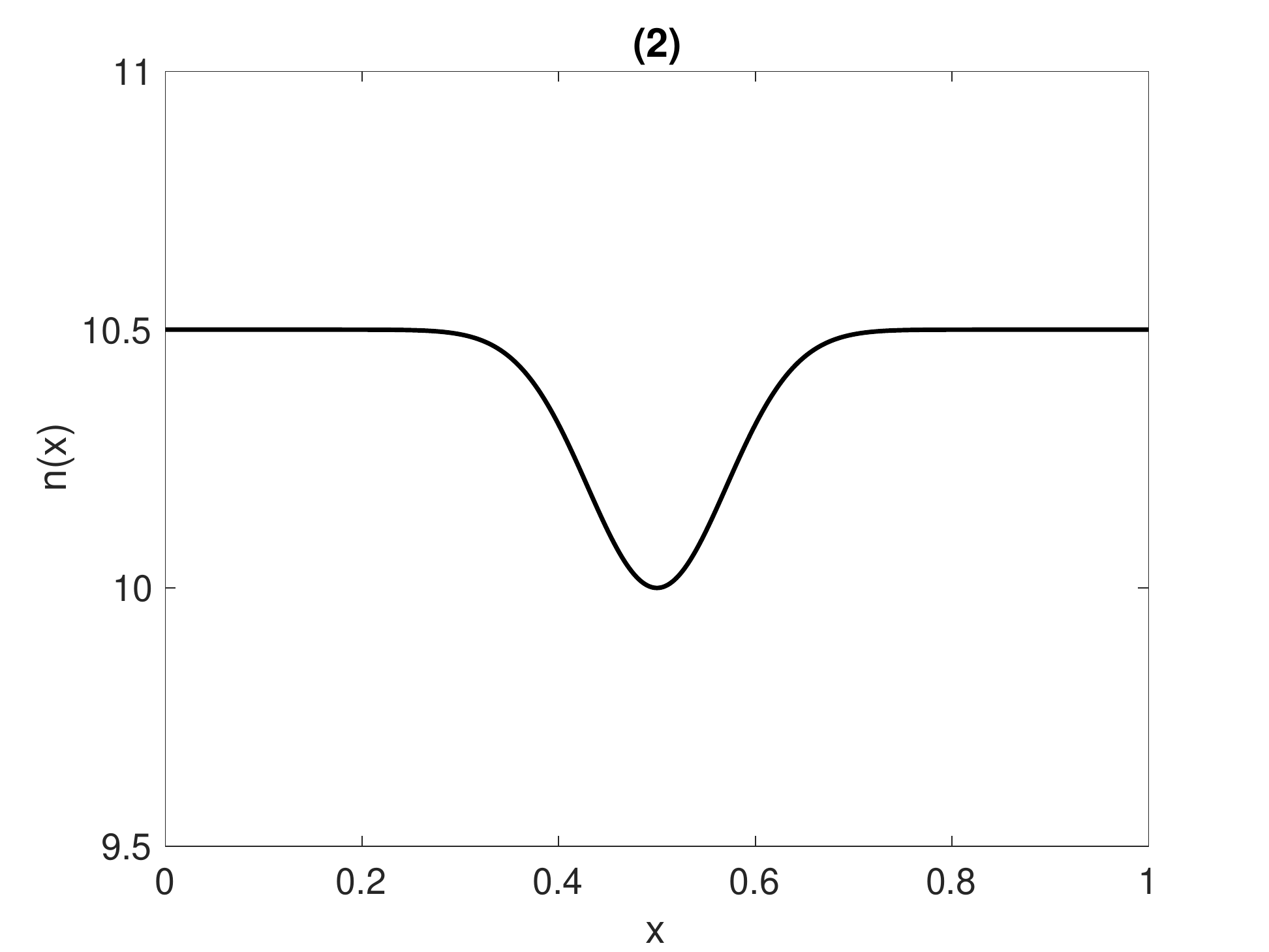}\tabularnewline
\end{tabular}
\par\end{centering}
\caption{Initial conditions for the simulations (1) and (2) depicted in Fig.
\ref{fig:inf_phase} with $\alpha=10$, $\beta=0.1$, $R=1$, and
$P=100$.} \label{fig:IC_spiral}

\end{figure}

\begin{figure}[H]
\begin{centering}
\begin{tabular}{cc}
\includegraphics[scale=0.32]{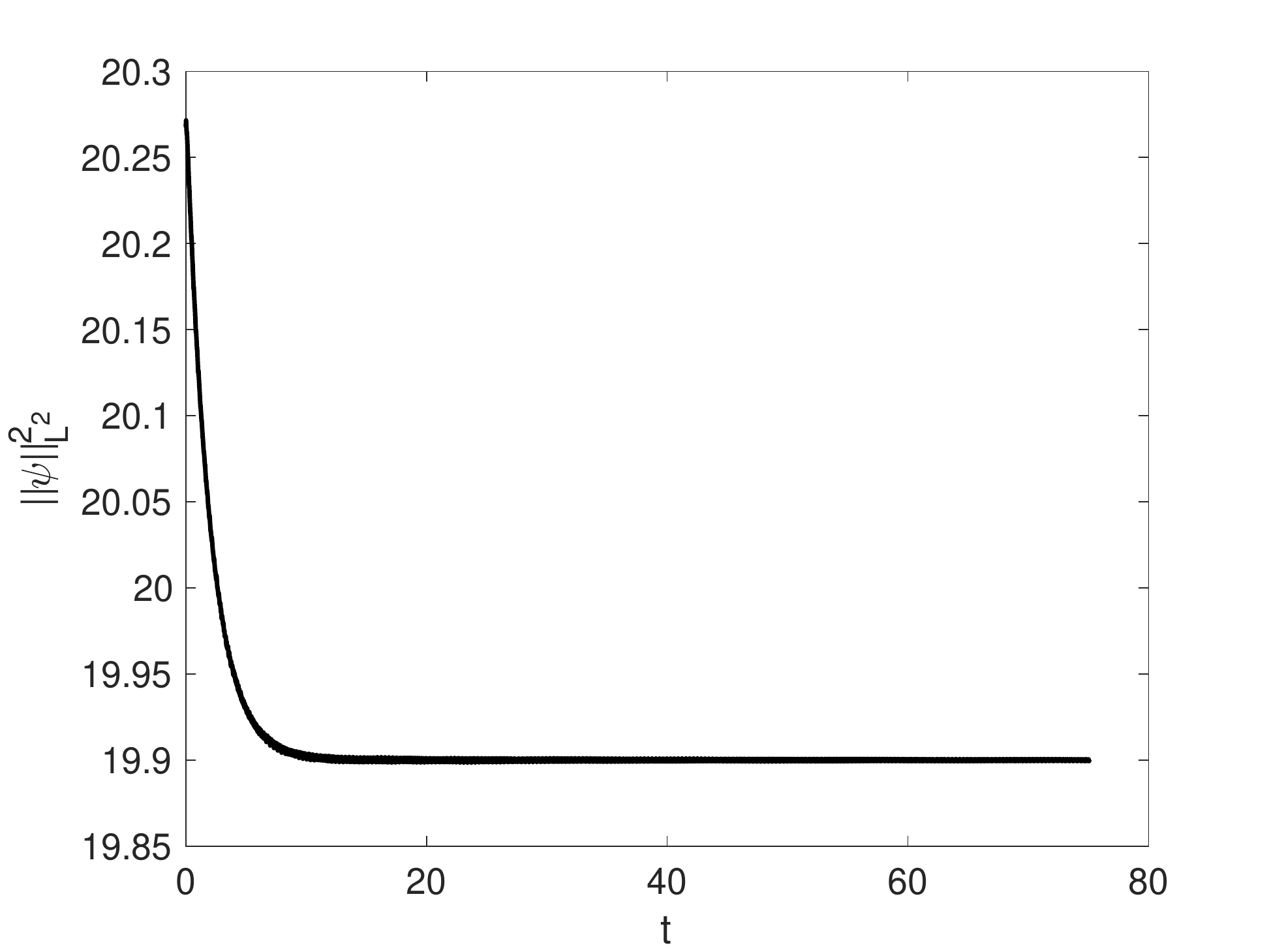} & \includegraphics[scale=0.32]{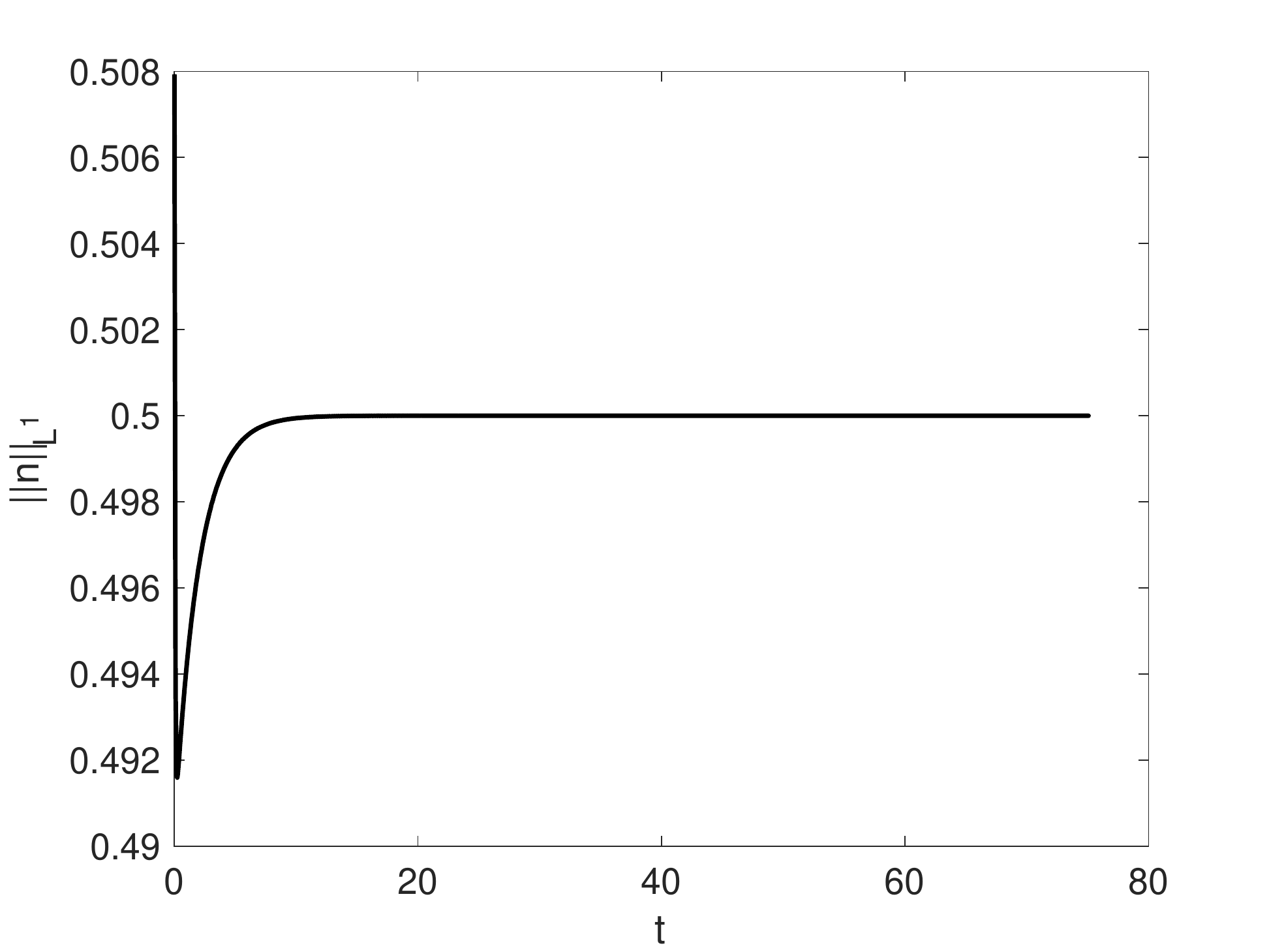}\tabularnewline
\end{tabular}
\par\end{centering}
\caption{Evolution of the (square of the) $L^{2}-$norm of $\psi$ (left) and
the $L^{1}-$norm of $n$ (right) corresponding to the simulation
of the perturbation of the stationary solution for $\alpha=0.5$,
$\beta=0.1$, $P=10$, and $R=1$. }\label{fig:L2_L1_exp2}
\end{figure}

Fig. \ref{fig:L2_L1_exp2} shows the evolution of the (square of the)
$L^{2}-$norm of $\psi$ and the $L^{1}-$norm of $n$
corresponding to the initial data depicted in Fig. \ref{fig:IC_spiral}, and $\alpha=0.5$, $\beta=0.1$, $P=10$, and $R=1$. Moreover,
Fig. \ref{phase_exp2} displays the plot of the $L^{2}-$norm of $\psi$
vs the $L^{1}-$norm of $n$ corresponding to: (0), the simulation shown in Fig.
\ref{fig:L2_L1_exp2}; (1) and (2), the simulations with the initial
conditions shown in Fig. \ref{fig:IC_node}. Like in the previous
case, it is interesting to notice the similarity with space-homogeneous
solutions: Fig. \ref{phase_exp2} resembles an asymptotically stable
node. 

For the case $\beta \gg \alpha$ we have the corresponding phase-space plot represented in Fig (\ref{phase_beta_mgt_alpha}).

\begin{figure}[H]
\begin{centering}
\includegraphics[scale=0.4]{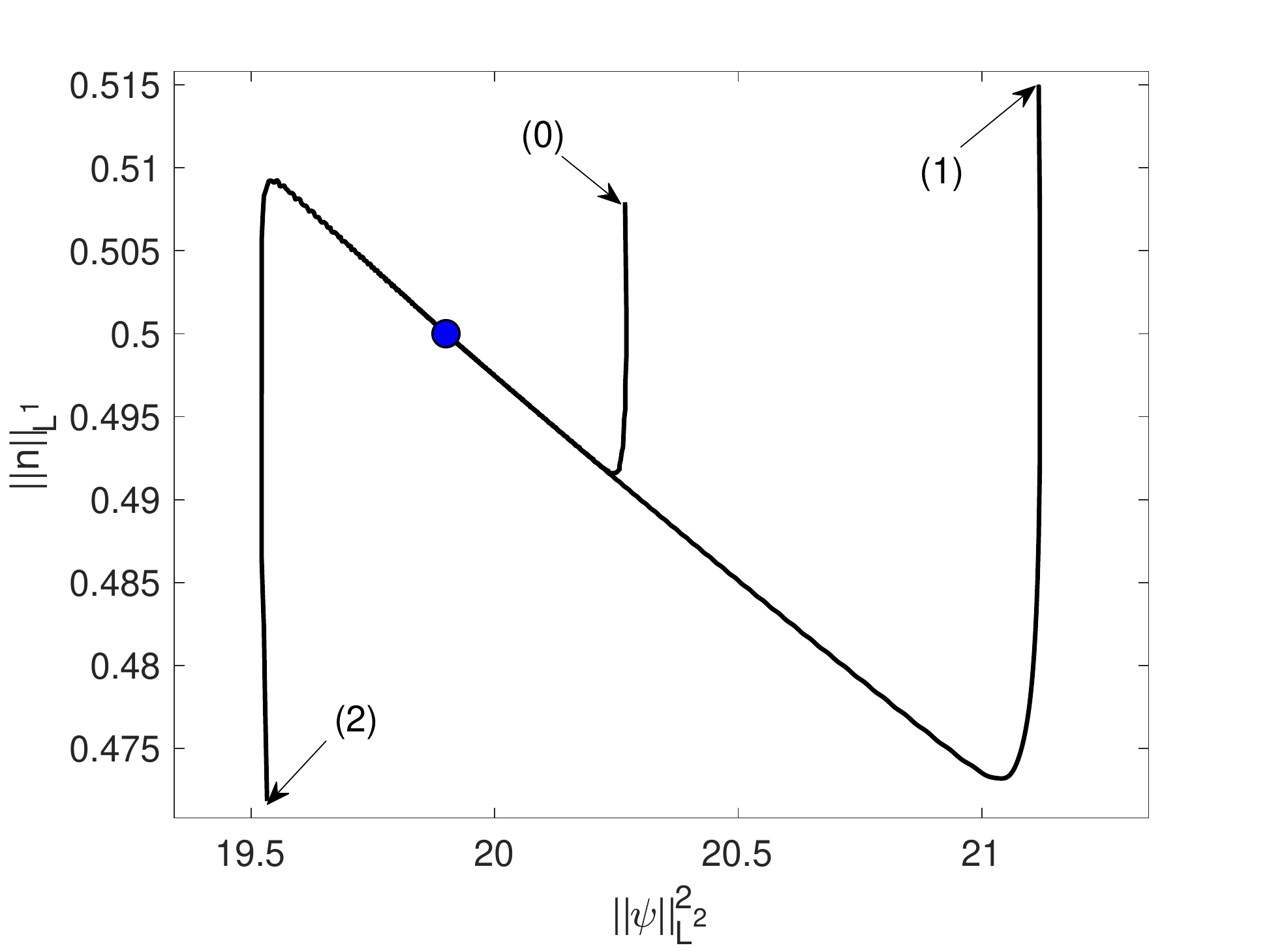}
\par\end{centering}
\caption{$L^{2}-$norm of $\psi$ vs $L^{1}-$norm of $n$ corresponding to: (0), the simulation
shown in Fig. \ref{fig:L2_L1_exp2}; (1) and (2), the simulations
with the initial conditions shown in Fig. \ref{fig:IC_node}. The equilibrium point is represented by the circle. Cf.
Fig. \ref{fig:node}. }\label{phase_exp2}
\end{figure}

\begin{figure}[H]
\begin{centering}
\begin{tabular}{cc}
\includegraphics[scale=0.3]{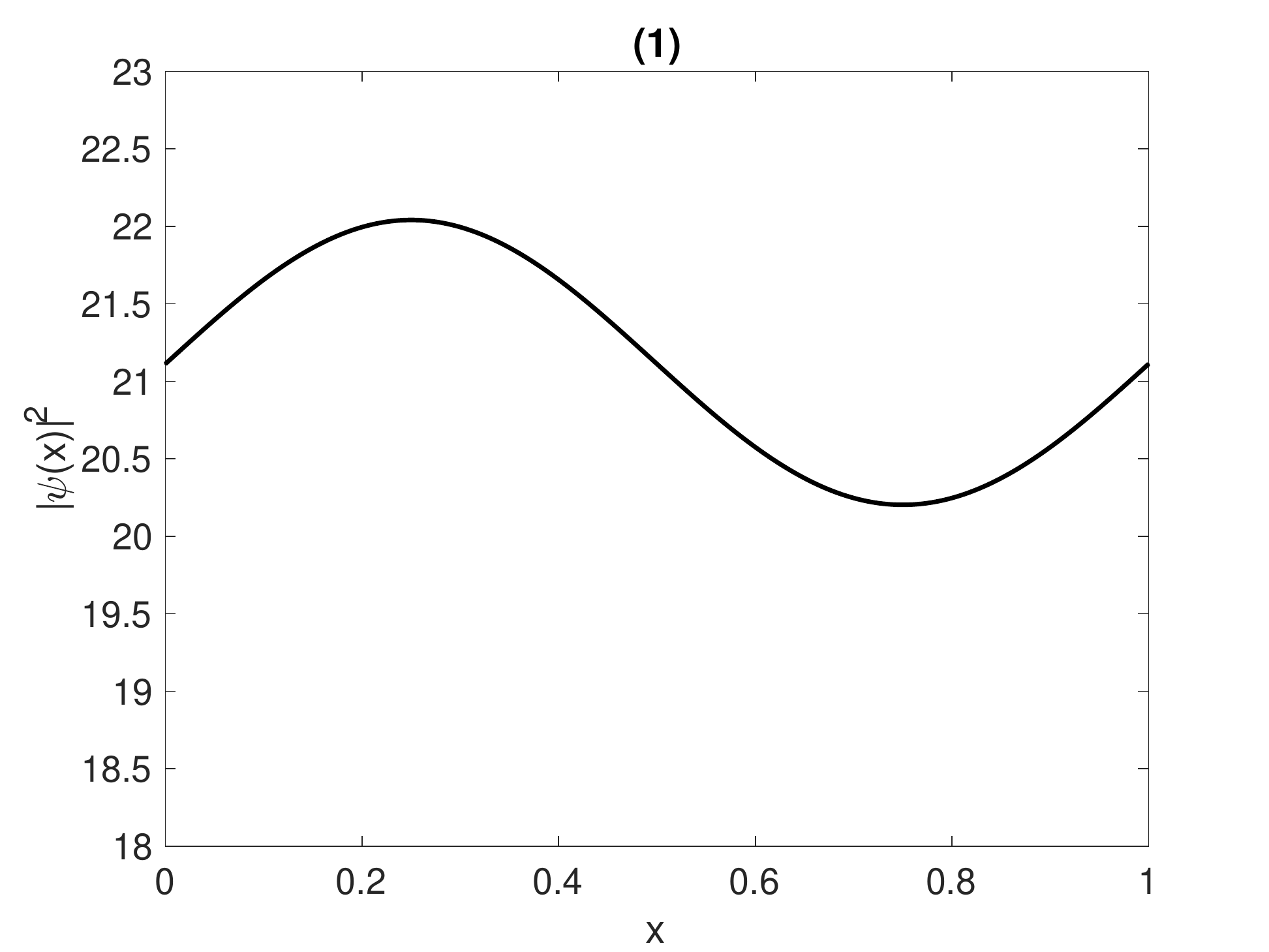} & \includegraphics[scale=0.3]{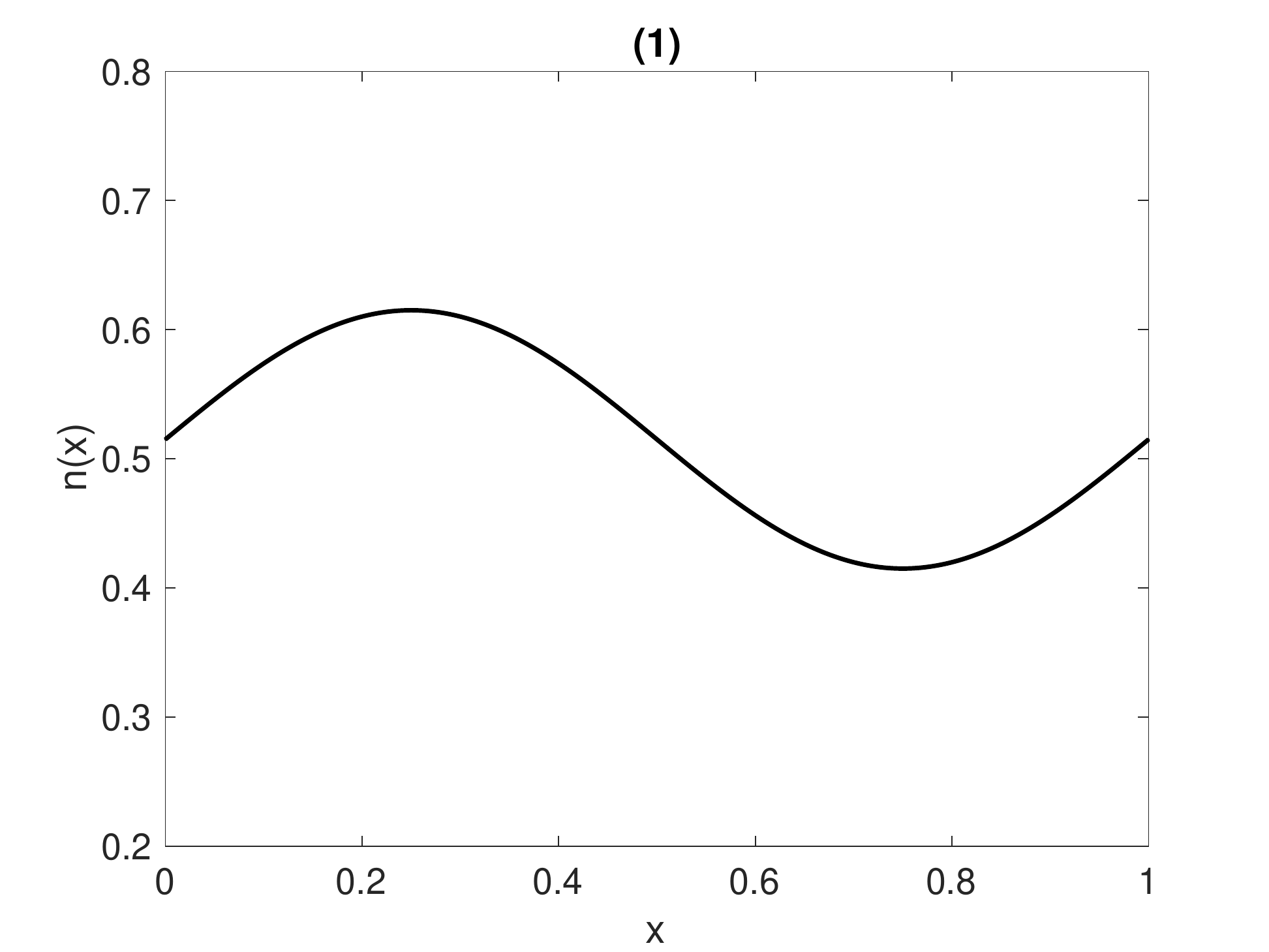}\tabularnewline
\includegraphics[scale=0.3]{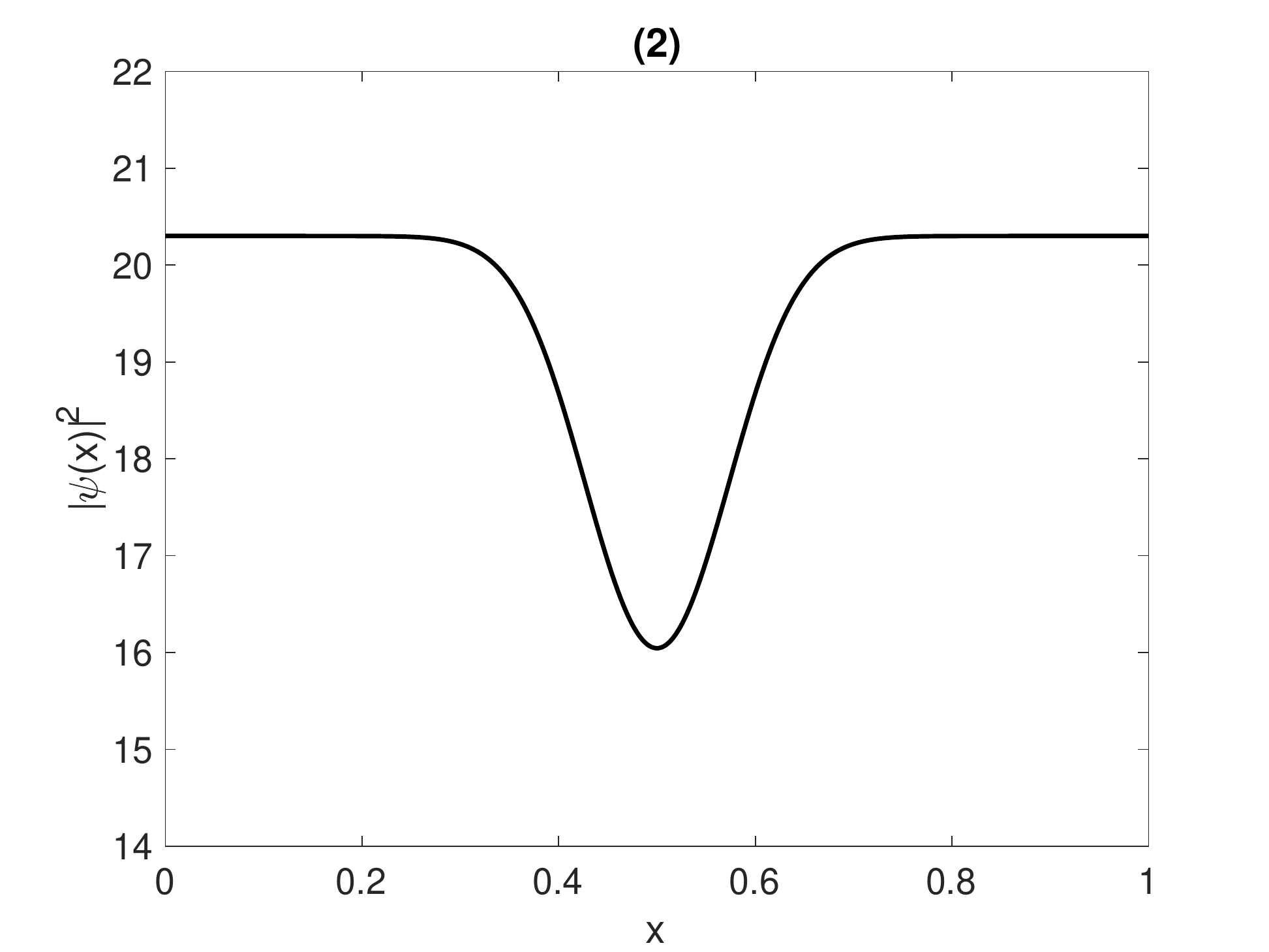} & \includegraphics[scale=0.3]{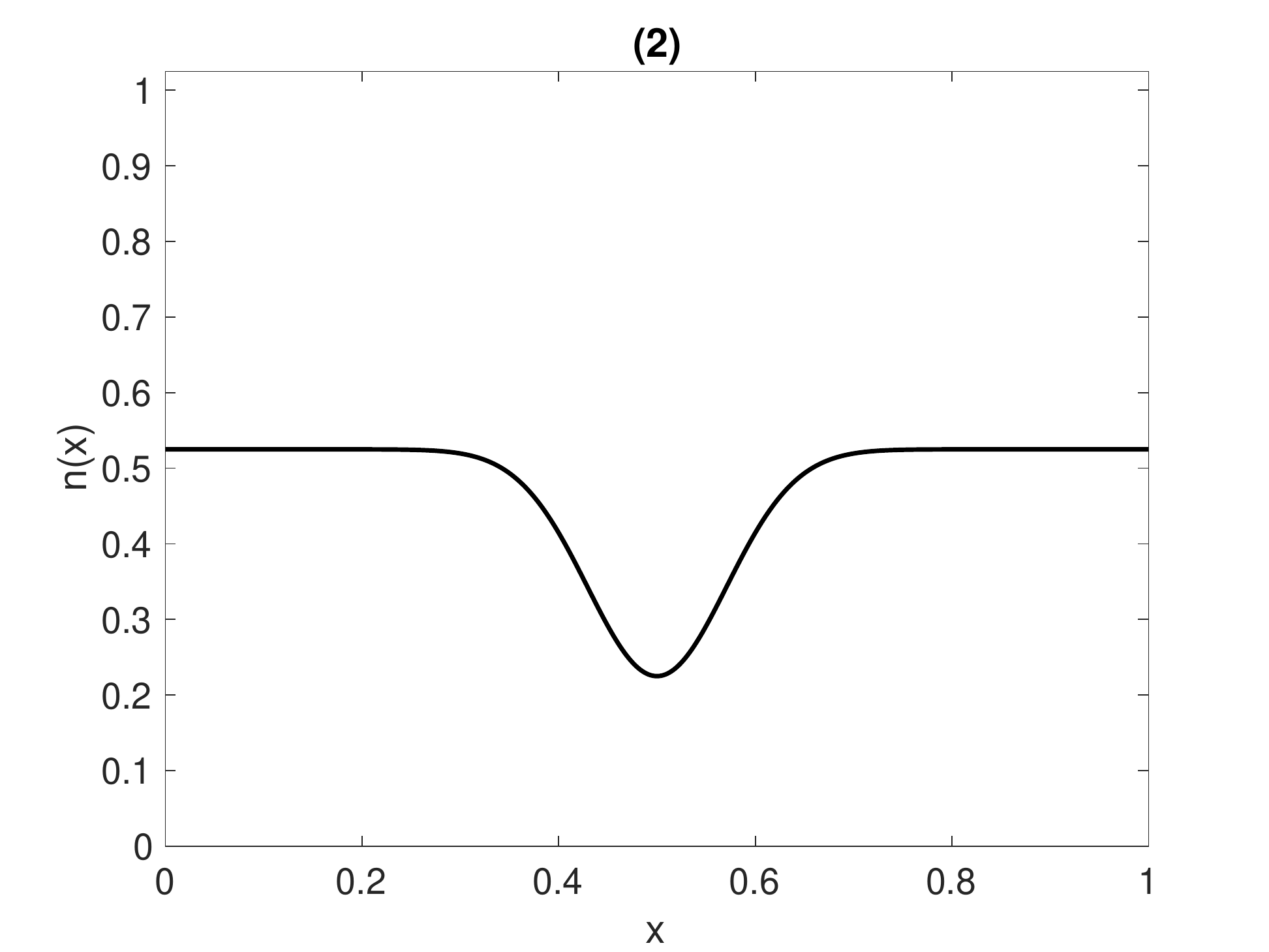}\tabularnewline
\end{tabular}
\par\end{centering}
\caption{Initial conditions for the simulations (1) and (2) depicted in Fig.
\ref{phase_exp2} with $\alpha=0.5$,
$\beta=0.1$, $P=10$, and $R=1$}\label{fig:IC_node}

\end{figure}

\begin{figure}[H]
\begin{centering}
\includegraphics[scale=0.5]{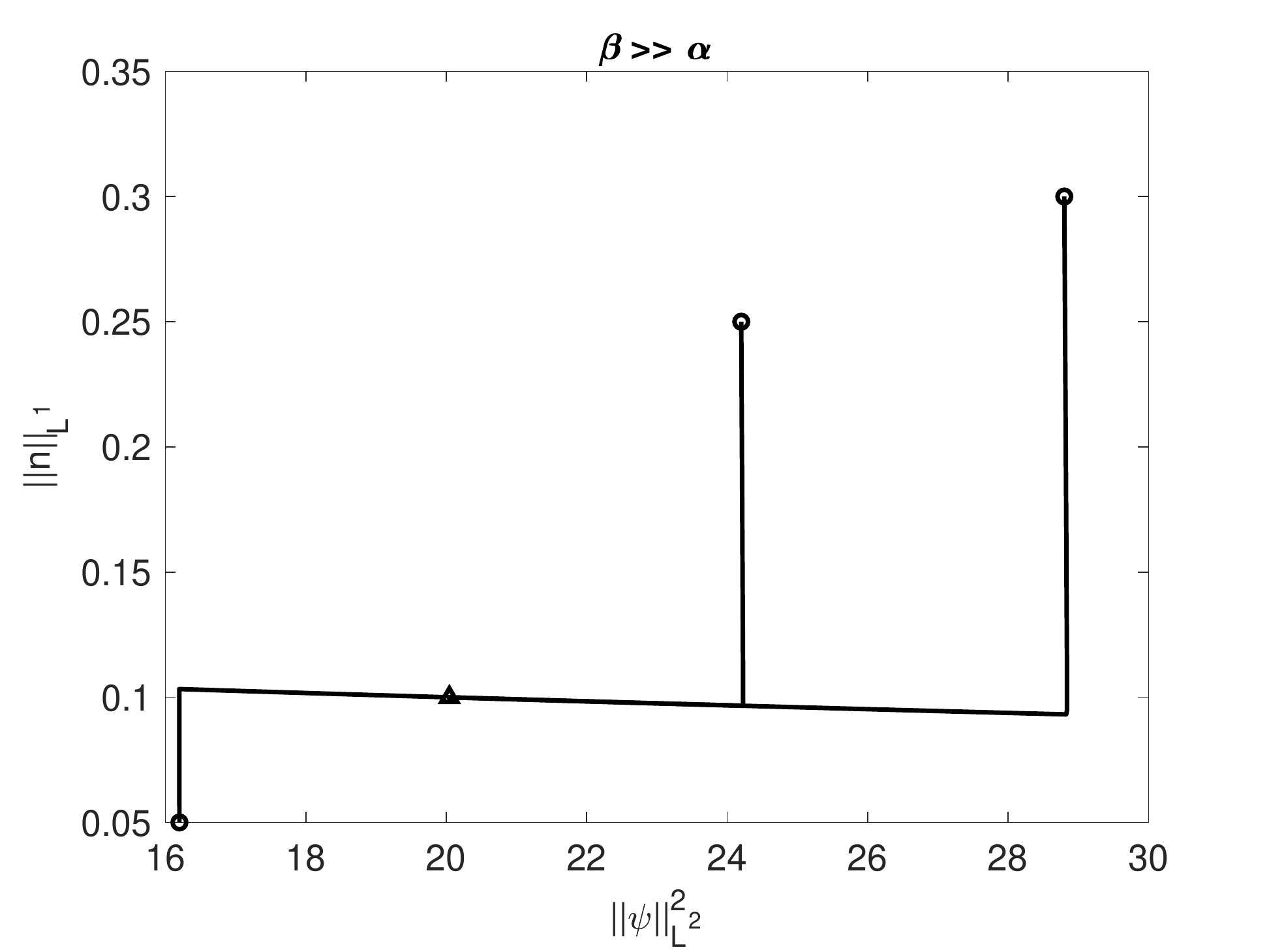}
\par\end{centering}
\caption{$L^{2}-$norm of $\psi$ vs $L^{1}-$norm of $n$ corresponding to the simulations with $\alpha=0.1$, $\beta=100$, $P=12$,
$R=1$, and the initial conditions indicated with the circles.  The equilibrium point is represented by the triangle.}\label{phase_beta_mgt_alpha}
\end{figure}

Finally, we turn to the case with vanishing condensate, i.e. $PR-\alpha\beta<0$: Fig. \ref{fig:Phase_zero}
shows the $L^{2}-$norm of $\psi$ vs the $L^{1}-$norm of $n$ corresponding
to the numerical simulations with $\alpha=10$, $\beta=10$, $P=1$,
$R=1$, and various initial conditions, indicated by circles. Notice
that Fig. \ref{fig:Phase_zero} is similar to an asymptotically stable
node.

\begin{figure}[H]
\begin{centering}
\includegraphics[scale=0.5]{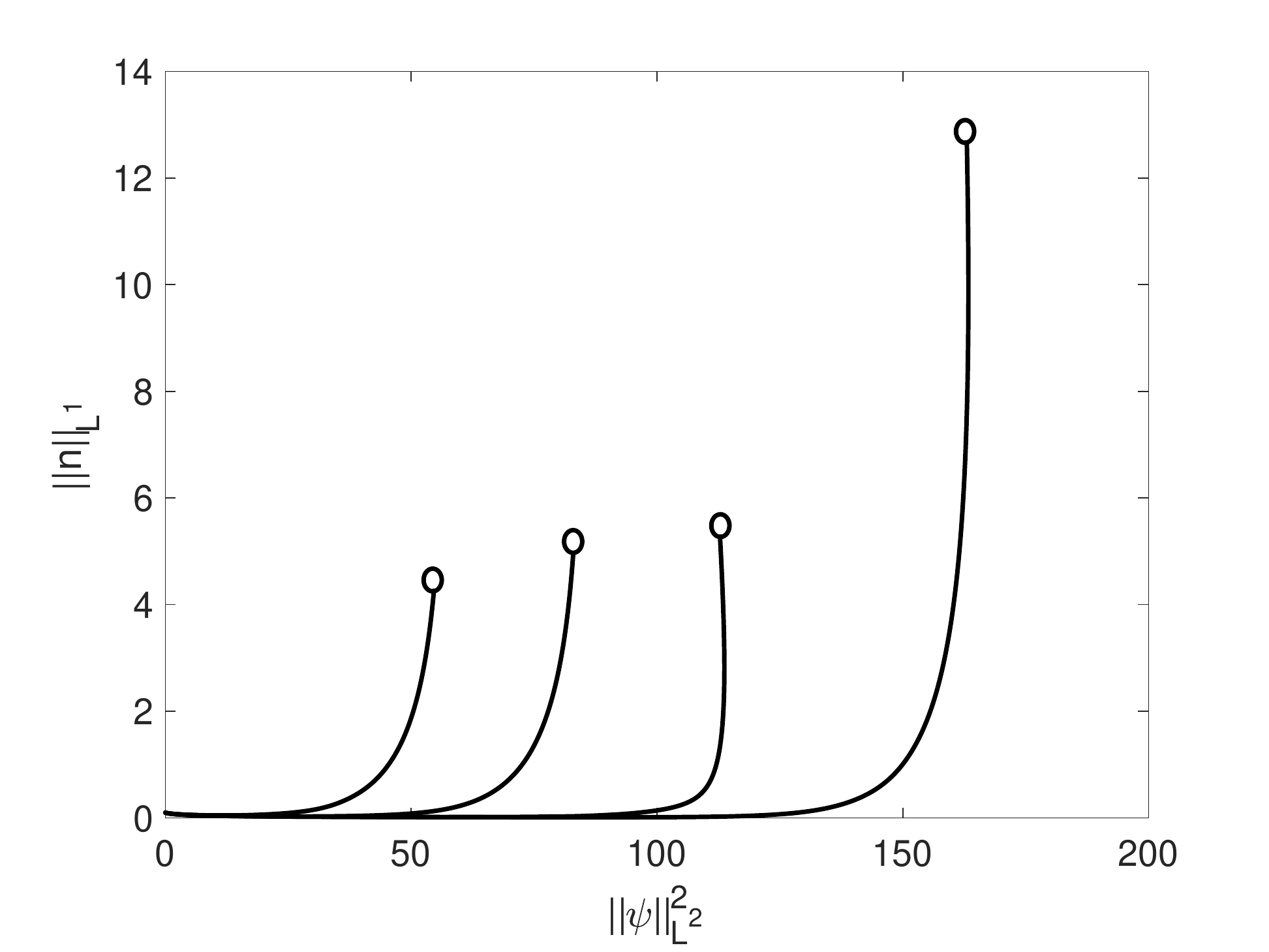}
\par\end{centering}
\caption{$L^{2}-$norm of $\psi$ vs $L^{1}-$norm of $n$ corresponding to
the numerical simulations with $\alpha=10$, $\beta=10$, $P=1$,
$R=1$; the initial conditions are indicated by circles. }\label{fig:Phase_zero}
\end{figure}

 
\section{The adiabatic regime}\label{sec:adiabatic}

In this last section, we look at a particular limiting case, called the {\it adiabatic regime}, cf \cite{BM}. It allows to reduce the 
full model \eqref{GP} to a single equation under the assumption that the reservoir density, $n$, adiabatically follows the change of $|\psi|^2$. 
Formally, one considers the limit $\varepsilon \to 0$ and simply drops the time derivative in the second equation of \eqref{GP}. This allows one to 
rewrite the exciton-density, $n$, via
\begin{equation}\label{nG}
n(t,x) = \frac{P}{\beta+ R|\psi(t,x)|^2}.
\end{equation}
Plugging this into the equation for $\psi$ yields a damped-driven Gross-Pitaevskii equation of the form
\begin{equation}\label{GPD}
i\partial_t\psi=-\frac12\partial_x^2\psi+g|\psi|^2\psi+ \frac{\lambda P \psi}{\beta+ R|\psi |^2} +\frac{i}{2}\left(\frac{PR}{\beta+ R|\psi|^2}-\alpha\right)\psi,
\end{equation}
subject to initial data $\psi |_{t=0} = \psi_0(x)$. 

\begin{remark}
The adiabatic model \eqref{GPD} shares certain similarities with an alternative mean-field equation for exciton-polariton condensates 
introduced in \cite{KeBe} (see also \cite{Si} for a numerical study). The main difference seems to be that in \eqref{GPD}, the damping is linear $\propto \alpha$ 
and the driving (or pumping) is nonlinear, while 
in the model in \cite{KeBe, Si} it is the other way around.
\end{remark}

In the following, we shall derive the result that will allow us to conclude global in-time existence of solution $\psi(t, \cdot) \in H^1(\T)$ of \eqref{GPD}. 
In particular, we shall connect the system \eqref{GP} and the adiabatic equation \eqref{GPD} by 
using the estimates derived in Section \ref{sec:a-priori} and the Aubin-Lions compactness lemma.

\begin{proposition}\label{prop:AdEps}
For $\varepsilon>0$ denote by $(\psi^\varepsilon,n^\varepsilon) \in C([0, \infty); \mathcal H^1(\mathbb T))$ 
the unique global solution of \eqref{GP} subject to initial data $(\psi_0, n_0)\in \mathcal H^1(\mathbb T)$. 
Then, up to extraction of a suitable subsequence, we have for all $T\ge 0$:
\[
\psi^{\varepsilon} \stackrel{\varepsilon \rightarrow 0_+
}{\longrightarrow} \psi\textrm{ in }C([0,T]\times \mathbb{T}),
\]
as well as 
\[
n^{\varepsilon} \stackrel{\varepsilon \rightarrow 0_+
}{\longrightarrow} n \textrm{ in $L^{\infty}((0,T)\times\mathbb{T}))$ weak-${\ast}$}
\]
In addition, for any $T>0$,  $\psi \in C([0,T], H^1(\mathbb T))$ is the unique mild solution of \eqref{GP} with initial data $\psi_0\in H^1(\T)$.
\end{proposition}

\begin{proof}
To pass to the adiabatic limit, $\varepsilon\rightarrow0$, in \eqref{GP} we
observe that the energy estimate given by Proposition \ref{prop:energy} implies that for all $T\ge 0$: $\psi^{\varepsilon}\in L^{\infty}\left(\left(0,T\right);{H}^{1}\left(\mathbb{T}\right)\right)$,
uniformly as $\varepsilon\rightarrow0$. As a consequence of Lemma \ref{lem:n_inf} and the continuous embedding $H^{1}\left(\mathbb{T}\right) \hookrightarrow
L^{\infty}\left(\mathbb{T}\right)$, we also have that for all $T>0$: $n^{\varepsilon},\psi^{\varepsilon}\in L^{\infty}\left((0,T)\times \mathbb{T}\right)$,
uniformly as $\varepsilon\rightarrow0$. 
From the first equation in
\eqref{GP}, we thus infer, by inspection, that $\partial_t \psi^{\varepsilon}\in L^{\infty}\left(\left(0,T\right);H^{-1}\left(\mathbb{T}\right)\right)$
uniformly as $\varepsilon\rightarrow0$. Now the Aubin-Lions lemma (see, e.g., \cite{Sim})
shows that, after extraction of a suitable subsequence,
\[
\psi^{\varepsilon}\stackrel{\varepsilon \rightarrow 0_+
}{\longrightarrow} \psi\textrm{ in }C([0,T]\times \mathbb{T}),\quad \forall \, T\ge 0,
\]
since $H^{1}\left(\mathbb{T}\right)$ embeds compactly into $C\left(\mathbb{T}\right)$
in one space dimension. Furthermore, we have that, after extraction of
a subsequence, $n^{\varepsilon}\rightharpoonup n$ in in $L^{\infty}((0,T)\times\mathbb{T}))$ weak-${\ast}$. 

To identify the limit, we multiply both equation in \eqref{GP} by test-functions $\varphi_1, \varphi_2\in C_0^\infty([0, \infty)\times \mathbb T)$ 
and pass to the limit $\varepsilon \to 0$ in the associated weak formulation (which is possible due to the strong convergence of $\psi^\varepsilon$ in the uniform topology). 
Note that we thereby lose the value of $n^\varepsilon \rightharpoonup n$ at $t=0$, to obtain
\[
0 = \iint \left(P-(R|\psi|^2+\beta)n\right)  \varphi_2 \, dt \, dx, \quad \forall \varphi_2 \in C_0^\infty([0, \infty)\times \mathbb T),
\]
i.e., the distributional reformulation of \eqref{nG}. 
This shows that, for all $T>0$, the limiting pair $(\psi, n)$ is a distributional solution of \eqref{GPD}, after $n$ has been computed via \eqref{nG}. In addition, we have that 
the limit $\psi  \in C([0,T], H^1(\mathbb T))$ has finite energy, $E(t)$.
A standard fixed point argument, similar to the one outlined in the appendix, allows us to obtain a unique local in-time solution, $\psi$, in the same 
class of $H^1$-solutions with finite energy. Since the latter is unique, it must coincide with the limiting function $\psi$ obtained before, which exists for all times $T>0$. 
In summary, this yields a unique global in-time solution of \eqref{GPD} with finite energy.
\end{proof}

\begin{remark}
For a sequence of solutions on the two- or three-dimensional
torus and on a time interval $\left(0,T\right)$, the Aubin-Lions
compactness argument gives weaker results, i.e., for all $T\ge 0$:
\[
\psi^{\varepsilon}\stackrel{\varepsilon \rightarrow 0_+
}{\longrightarrow}\psi\textrm{ in }C\left(\left[0,T\right];L^{q}\left(\mathbb{T}\right)\right)
\]
with $1\le q<\infty$ for $d=2$ and $1\le q<6$ for $d=3$. With
$n^{\varepsilon}\rightharpoonup n$ in $L^{\infty}((0,T)\times\mathbb{T}))$
weak-${\ast}$, this is sufficient to identify the limit $\left(\psi,n\right)$
as an energy-bounded solution of the adiabatic equation \eqref{GPD}. 
\end{remark}

On the other hand, we can directly pass to the limit $\varepsilon \to 0$ in the estimate stated in Lemma \ref{lem:mas} to conclude that 
the solution of \eqref{GPD} satisfies:
\[
\lVert \psi \rVert_{L^2}^2 \le  \left(\lVert \psi_0 \rVert_{L^2}^2 - \frac{P|\T|}{\alpha} \right) e^{-\alpha t} + \frac{P|\T|}{\alpha},
\]
and thus
\[
\limsup_{t \rightarrow + \infty} \, \lVert \psi(t,\cdot) \rVert_{L^2}^2 \le \frac{P|\T|}{\alpha}.  
\]

Finally, a direct computation also shows that in the case of vanishing condensate density the solution exponentially converges to zero. More precisely, we have:
\begin{lemma}
Let $\psi \in C([0, \infty);H^1(\T))$ be a solution of \eqref{GPD} and assume that $PR -\alpha \beta <0$. Then
\[
\| \psi (t, \cdot) \|_{L^2} \le \| \psi_0 \|_{L^2} e^{-\kappa t},\quad \forall t\ge 0,
\]
where $\kappa := \frac{\alpha - PR}{2\beta}>0$.
\end{lemma}

\begin{proof}
As before, we multiply \eqref{GPD} by $\overline{\psi}$, integrate over the spatial domain, and take the imaginary part, to obtain
\begin{equation*}\label{GPDA}
\frac{1}{2} \, \frac{\partial}{\partial t} \int_\T |\psi|^2 dx = -\frac{1}{2} \, \im \, \int_\T \partial_x^2 \psi \overline{\psi} \, dx 
+ \frac{PR}{2} \int_\T \frac{|\psi|^2}{\beta + R|\psi|^2} \, dx - \frac{\alpha}{2} \int_\T |\psi|^2\, dx.
\end{equation*}
Here, the first term on the right hand side vanishes after an integration by parts, and since $R >0$, we have
\begin{align*}
\frac{\partial}{\partial t} \int_\T |\psi|^2 \, dx &=  PR \int_\T \frac{|\psi|^2}{\beta + R|\psi|^2} \, dx - \alpha \int_\T  |\psi|^2 \, dx \\
&\le \left(\frac{P R}{\beta} - \alpha \right) \int_\T |\psi|^2 \, dx.
\end{align*}
This directly yields the exponential bound stated above. 
\end{proof}

As before, one may look for spatially homogenous solutions, $\psi = \sqrt{\rho} e^{i \phi}$, of \eqref{GPD}, which yields the following ordinary differential equation for the particle density:
\begin{equation}\label{odea}
\dot \rho = \frac{1}{2}\left(\frac{PR}{\beta+ R\rho}-\alpha\right)\rho,\quad \rho|_{t=0} = \rho_0>0.
\end{equation}
This equation can be solved by a lengthy, but straightforward computation, to give:

\begin{lemma} For every $\rho_0>0$, \eqref{odea} admits a unique solution $\rho(t)\in C^1[0,\infty)$ which is positive. 
Moreover, there are two steady states given by
\[
\rho_1^* = 0, \quad \text{and} \quad  \rho_2^*= \frac{PR-\beta \alpha}{R\alpha},
\]
which are consistent with the ones found in Theorem \ref{thm:ODE}.
\end{lemma}


\bibliographystyle{amsplain}

\begin{thebibliography}{99}

\bibitem{ACS} P. Antonelli, R. Carles, and C. Sparber. {\it On nonlinear Schr\"{o}dinger type equations with nonlinear damping.} Int. Math. Res. Not. {\bf 2015} (2015), no. 3, 740--762.

\bibitem{Ba} W. Bao, D. Jaksch, and P. Markowich, {\it Numerical solution of the Gross--Pitaevskii equation for Bose--Einstein condensation}, J. Comput. Phys. {\bf 187} (2003), 318--342.
 
\bibitem{Ba2} W. Bao, S. Jin, and P. Markowich, {\it On time-splitting spectral approximations for the Schr{\"o}dinger equation in the semiclassical regime}, J. Comput. Phys. {\bf 175} (2002), 487--524.

\bibitem{Bi} B. Bid\'egaray, {\it The Cauchy problem for Schr\"odinger-Debye equations}, Math. Models Methods Appl. Sci. {\bf 10} (2000), 307--315.

\bibitem{BM} N. Bobrovska and M. Matuszewski, {\it Adiabatic approximation and fluctuations in exciton-polariton condensates}, Phys. Rev. B {\bf 92} (2015), 035311, 7pp.

\bibitem{BOM} N. Bobrovska, E.~A. Ostrovskaya, M. Matuszewski, {\it Stability and spatial coherence of nonresonantly pumped exciton-polariton condensates}, Phys. Rev. B {\bf 90}, 
(2014), 205304, 6pp.

\bibitem{BrMo} A. Bramati and M. Modugno, {\it Physics of Quantum Fluids. New Trends and Hot Topics in Atomic and Polariton Condensates}. Springer Series in Solid State Sciences vol. 177, 
Springer Verlag, 2013.

\bibitem{CCFHKR} R. Carretero-Gonz\'alez, J. Cuevas-Maraver, D.~J. Frantzeskakis, T.~P. Horikis, P.~G. Kevrekidis, A.~S. Rodrigues, 
\emph{A Korteweg-de Vries description of dark solitons in polariton superfluids}, Phys. Lett. A {\bf 381} (2017), no. 45, 3805--3811.

\bibitem{Car} I. Carusotto and C. Ciuti, {\it Quantum fluids of light}, Rev. Mod. Phys. {\bf 85} (2013), 299--366.

\bibitem{Caz} T. Cazenave, \emph{Semilinear Schr\"odinger equations}. Courant Lecture Notes in Mathematics vol. 10, AMS, Providence, RI, 2003. 

\bibitem{Com} E. Compaan, {\it Smoothing and global attractors for the Majda-Biello system on the torus}, Differential Integral Equ. {\bf 29} (2016), 269--308.

\bibitem{Com2} E. Compaan, {\it Smoothing for the Zakharov and Klein--Gordon--Schr\"odinger Systems on Euclidean Spaces}, SIAM J. Math. Anal. {\bf 49} (2017), 4206--4231.

\bibitem{COS} A. J. Corcho, F. Oliveira, and J. D. Silva, {\it Local and global well-posedness for the critical Schr\"odinger-Debye system}, Proc. Amer. Math. Soc. {\bf 141} (2013), 3485--3499.

\bibitem{EMNT} M. B. Erdo{\u{g}}an, J. L. Marzuola, K. Newhall, and N. Tzirakis, {\it The structure of global attractors for dissipative Zakharov systems with forcing on the torus}, SIAM J. Appl. Dyn. Syst. {\bf 14} (2015), 1978--1990.

\bibitem{ET} M. B. Erdo{\u{g}}an and N. Tzirakis, {\it Smoothing and global attractors for the Zakharov system on the torus}, Analysis \& PDE {\bf 6} (2013), 723--750.

\bibitem{HDTT} H. Haug, T.~D. Doan, D.~B. Tran Thoai, {\it Quantum kinetic derivation of the nonequilibrium Gross-Pitaevskii equation for nonresonant 
excitation of microcavity polaritons}, Phys. Rev. B {\bf 89}, (2014), 155302, 11pp.

\bibitem{Kasp} J. Kasprzak, et al., {\it Bose-Einstein condensation of exciton polaritons}, Nature {\bf 443} (2006), 409.

\bibitem{KeBe} J. Keeling and N.G. Berloff, {\it Spontaneous rotating vortex lattices in a pumped decaying condensate}, Phys. Rev. Lett. {\bf100} (2008), no. 25, 250401, 4pp.

\bibitem{Per} L. Perko, {\it Differential equations and dynamical systems}, Texts in Applied Mathematics vol. 7, Springer Verlag, 2013. 

\bibitem{Si} J. Sierra, A. Kasimov, P. Markowich, and R.M. Weish{\"a}upl, {\it On the Gross--Pitaevskii Equation with Pumping and Decay: Stationary States and Their Stability}, J. Nonlin. Sci. 
{\bf 25} (2015), 709--739.

\bibitem{Sim} J. Simon, {\it Compact sets in the space $ L^{p} (0, T; B) $},  Ann. Mat. Pura Appl. {\bf 146} (1986), 65--96.

\bibitem{Tao} T. Tao. {\it Nonlinear dispersive equations}, Local and global analysis, CBMS Regional Conference Series in Mathematics {\bf 106}, 
American Mathematical Society, Providence, RI, 2006.

\bibitem{WoCa} M. Wouters and I. Carusotto, {\it Excitations in a nonequilibrium Bose-Einstein condensate of exciton polaritons}. {Phys. Rev. Lett.} {\bf 99} (2007), 140402, 4pp.

\bibitem{WCC} M. Wouters, I. Carusotto, and C. Ciuti, {\it Spatial and spectral shape of inhomogeneous nonequilibrium exciton-polariton condensates}, Phys. Rev. B {\bf 77} (2008), 115340, 7pp.

\bibitem{ZP} V.F. Zaitsev and A.D. Polyanin, {\it Handbook of Exact Solutions for Ordinary Differential Equations}. CRC press, 2002.


\end{thebibliography}


\appendix

\section{Local existence of smooth solutions}\label{sec:LWPH1}

In this appendix, we shall give the proof of Proposition \ref{prop:local}. For simplicity, we 
denote
\[
U(t,x)\equiv 
\left (
\begin{array}{c}
\psi(t,x) \\
n(t,x)
\end{array} \right)
\]
for $t\in [0,\infty)$ and $x \in \T$. Using this, we can rewrite \eqref{GP}-\eqref{ini} in the following form
\begin{equation}\label{GPU}
\partial_t U = AU + f(U), \quad U|_{t=0} =U_0(x),
\end{equation}
with $U_0 = (\psi_0, n_0)^\top$, and 
\begin{equation}\label{A}
A = \left(
\begin{array}{cc}
\frac{i}{2}\partial_x^2 & 0 \\
0 & -\beta
\end{array} \right), 
\end{equation}
as well as
\begin{equation}\label{f}
f(U) = \left(
\begin{array}{cc}
-ig |\psi|^2 \psi -i\lambda n \psi + \frac{1}{2}(Rn-\alpha)\psi \\
P -R|\psi|^2n
\end{array} \right).
\end{equation}
Note that $f(0) = (0,P)^\top$. By means of Duhamel's formula, we can rewrite \eqref{GPU} as an integral equation for $U$, i.e.
\begin{equation} \label{DF}
U(t,x) = e^{tA}U_0(x) + \int_0^t e^{(t-\tau)A}f(U(\tau, x)) \, d\tau\equiv \Phi(U)(t).
\end{equation}
We shall prove that for some (sufficiently small) time $t>0$, $\Phi$ is a contraction mapping on $\mathcal{H}^s = H^s(\T) \oplus H^s(\T)$, provided $s>\frac{1}{2}$.
To this end, the following lemma is key.

\begin{lemma}\label{lem:lip}
For $s>\frac{1}{2}$, the nonlinear map $f:\mathcal{H}^s \rightarrow \mathcal{H}^s$ given by \eqref{f} is locally Lipschitz continuous in $U$, uniformly for $t\in [0, \infty)$.
\end{lemma}

\begin{proof}
Let $U=(\psi, n)^\top$, $V=(\phi, m)^\top$ and suppose that $U,V \in \overline{B_M(0)}\subset \mathcal{H}^s$, for some $M >0$. Then, by triangle inequality
\begin{align*}
& \lVert f(U)-f(V) \rVert_{\mathcal{H}^s}^2= \\
&= \big \lVert -ig(|\psi|^2\psi - |\phi|^2 \phi) -i\lambda (n\psi - m \phi) + \frac{1}{2}[(Rn-\alpha)\psi - (Rm-\alpha)\phi] \big \rVert_{H^s(\mathbb{T})}^2 \\
&\quad + \lVert -R (|\psi|^2n - |\phi|^2m) \rVert_{H^s(\mathbb{T})}^2 \\
&\le g^2 \lVert |\psi|^2\psi - |\phi|^2 \phi \rVert_{H^s(\mathbb{T})}^2 + \lambda^2 \lVert n\psi - m \phi \rVert_{H^s(\mathbb{T})}^2 + \frac{1}{4} R^2 \lVert n\psi - m\phi \rVert_{H^s(\mathbb{T})}^2\\
&\quad  + \frac{1}{4}\alpha^2 \lVert \psi - \phi \rVert_{H^s(\mathbb{T})}^2 + R^2 \lVert (|\psi|^2n - |\phi|^2m) \rVert_{H^s(\mathbb{T})}^2.
\end{align*}
Repeated use of the facts that (i) $H^s(\mathbb{T})$ is an algebra for all $s > 1/2$, and (ii) polynomials of the form $\psi n -\phi m$ can be factored as
\[
\psi n -\phi m = \frac{1}{2}\big((\psi+\phi)(n-m)+(n+m)(\psi-\phi)\big),
\]
together with the assumption $\psi,\phi,n,m \in \overline{B_M (0)}$ yields
\begin{align*}
\lVert f(U)-f(V) \rVert_{\mathcal{H}^s}^2 & \le C_{M} \left( \lVert \psi -\phi \rVert_{H^s(\mathbb{T})}^2 +  \lVert n -m \rVert_{H^s(\mathbb{T})}^2 \right) \\
&= C_M \lVert U-V \rVert_{\mathcal{H}^s}^2,
\end{align*}
where $C_M$ is a constant depending only on $M$.
\end{proof}

With this lemma in hand, we can now give the proof of Proposition \ref{prop:local}:

\begin{proof}
We first prove existence and uniqueness. We define the Banach space $X := C([0,T);\mathcal{H}^s)$, where $T >0$ will be determined below. 
Let $\| U_0\|_{\mathcal H^s} \le M$ and consider the subspace
    \[
    K = \{U \in X : \lVert U \rVert_X \le 2M\},
    \]
    Then $K$ is a closed subspace of $X$, 
    so it is a complete metric space, and we can apply Banach's fixed point theorem, provided 
    $\Phi$ maps $K$ into itself and there exists $\theta \in (0,1)$ such that
    \[
    \lVert \Phi(U) - \Phi(V) \rVert_X \le \theta \lVert U-V \rVert_X, \quad \forall \, U,V \in K.
    \]
    
To this end, we first notice that, since $e^{-\frac{i}{2}t\partial_x^2}$ is a unitary group on every $H^s(\T)$, $s\in \R$, and $\beta >0$, we have 
the following bound on the linear time-evolution generated by \eqref{A}:
\[
\lVert e^{tA}U(t) \rVert_{\mathcal H^s} \le \lVert U(t) \rVert_{\mathcal H^s}, \quad \forall \, t\ge 0.
\]
It follows that
    \begin{align*}
    \lVert \Phi(U) \rVert_X  & = \sup_{0 \le t < T} \left \lVert e^{tA}U_0 + \int_0^t e^{(t-\tau)A}f(U(\tau)) \, d\tau \right \rVert_{\mathcal{H}^s} \\
    &\le \lVert U_0 \rVert_{\mathcal{H}^s} +  T \lVert f(U) \rVert_X,
    \end{align*}
where we have used the triangle inequality and Minkowski's inequality. We now invoke Lemma \ref{lem:lip}, which gives
\begin{equation}\label{Lest}
\lVert f(U)-f(V) \rVert_{X} \le C \lVert U-V \rVert_{X},
\end{equation}
so that, taking $V = 0$, we have, by triangle inequality
\begin{align*}
\lVert f(U) \rVert_{X} \le \| f(u) - f(0)\|_X + \| f(0)\|_X\le \lVert f(0) \rVert_{X} + C \lVert U \rVert_{X}.
\end{align*}
Note that in the case of a constant exciton creation rate $P>0$, we explicitly have:
\begin{align}\label{eq:pump}
    \lVert f(0) \rVert_X & \equiv \sup_{0 \le t \le T} \lVert (0,P)^\top \rVert_{\mathcal H^s} = \lVert P \rVert_{L^2(\T)} = P \sqrt{|\T|}.
    \end{align}
We use this, together with the assumption that $U \in K$, to obtain
    \begin{align*}
    \lVert \Phi(U) \rVert_X &\le M + T(\lVert f(0) \rVert_X + C \lVert U \rVert_X) \\
    &\le M + T\big (P \sqrt{|\T|} + 2CM \big),
    \end{align*}
    where $C$ is a constant depending on $M$. Therefore, choosing
    \[
    T = \frac{M}{2CM + P \sqrt{|\T|}},
    \]
    we have that $\Phi(K) \subseteq K$. To show that $\Phi$ is a contraction on $K$, we again use \eqref{Lest}, to obtain
    \begin{align*}
    \lVert \Phi(U) - \Phi(V) \rVert_{X} &= \sup_{0\le t\le T} \left \lVert \int_0^t e^{(t-\tau)A}(f(U(\tau))-f(V(\tau))) \, d\tau \right \rVert_{\mathcal{H}^s} \\
    &\le CT \lVert U-V \rVert_{X}\\
    &\le \frac{1}{2} \lVert U-V \rVert_X,
    \end{align*}
    with the same $T$ chosen above. Banach's fixed point theorem consequently implies that 
    there exists a unique fixed point $U \in K$ such that $\Phi(U) = U$. This $U$ is the unique solution of \eqref{DF} in $K$.
    
In fact, the solution $U$ is unique in $X$, not only in $K$. This is because our choice of $T$, together with the fact that we have chosen 
$ \lVert U_0 \rVert_{\mathcal{H}^s}\le M$, ensures that any solution $U \in X$ actually belongs to $K$. To see this, let $U$ be a solution of \eqref{DF}. Then we have
    \begin{align*}
    \lVert U \rVert_X &= \sup_{0 \le t \le T} \left \lVert  e^{tA}U_0 + \int_0^t e^{(t-\tau)A}f(U(\tau)) \, d\tau \right \rVert_{\mathcal{H}^s} \\
    &\le \lVert U_0 \rVert_{\mathcal{H}^s} + \sup_{0 \le t \le T} \int_0^t \lVert f(U(\tau))\rVert_{\mathcal{H}^s} \, d\tau \\
    &\le M + \frac{1}{2}\lVert U \rVert_X.
    \end{align*}
    It follows that $\lVert U \rVert_X \le 2M$, so $U \in K$.

Having obtained a unique local solution for $t < T$, we can now extend it (uniquely) to a maximal solution on some time interval $[0,T_{\mathrm{max}})$, 
where either (i) $T_{\mathrm{max}} = \infty$, or (ii) $T_{\mathrm{max}} < \infty$, and $$\lim_{t \rightarrow T_{\mathrm{max}}} \lVert U(t) \rVert_{\mathcal H^s} = \infty,$$ 
since otherwise the solution could be extended, by continuity, past $T_{\mathrm{max}}$, which is a contradiction. 

Finally, 
continuous dependence on initial data follows by a classical Gronwall-argument: Indeed, for $U_0, V_0\in \mathcal H^s$, we find
\begin{align*}
\| U(t, \cdot) - V(t, \cdot) \|_{\mathcal H^s} \le \| U_0 - V_0 \|_{\mathcal H^s}+ C\int_0^t \| U(\tau, \cdot) - V(\tau, \cdot)\|_{\mathcal H^s} \, d\tau,
\end{align*}
where $C>0$ is the same Lipschitz constant as before. Thus 
\[
\| U(t, \cdot) - V(t, \cdot) \|_{\mathcal H^s} \le  \| U_0 - V_0 \|_{\mathcal H^s} e^{Ct}, \quad \forall \, 0\le t \le T< T_{\rm max},
\]
which proves the claim.
\end{proof}

This proof straightforwardly extends to the higher dimensional setting by requiring $s>d/2$.


\end{document}